\documentclass[10pt]{article}
\usepackage{amsmath, amsthm, amssymb, amsfonts, mathrsfs, mathtools}
\usepackage{graphicx}
\usepackage{makeidx}
\usepackage[left=2cm,right=2cm,top=2cm,bottom=2cm]{geometry}
\usepackage{caption}
\usepackage{subcaption}
\usepackage[section]{placeins}
\usepackage{enumitem}
\usepackage{tikz}
\usepackage{stmaryrd}

\newcommand{\Mod}[1]{\ (\mathrm{mod}\ #1)}

\usetikzlibrary{decorations.pathreplacing}
\tikzstyle{vertex}=[auto=left,circle,draw=black,fill=white, inner sep=1.5]

\usepackage{hyperref}
\hypersetup{colorlinks=true, citecolor={blue}, linkcolor=blue, filecolor=magenta, urlcolor=cyan}

\newtheorem{theorem}{Theorem}[section]

\newtheorem{lema}[theorem]{Lemma}
\newtheorem{corollary}{Corollary}[theorem]

\newtheorem{ex}{Example}[section]

\title{HS-integral and Eisenstein integral mixed Cayley graphs over abelian groups}

\author{ Monu Kadyan and Bikash Bhattacharjya\\
Department of Mathematics\\
Indian Institute of Technology Guwahati, India\\
monu.kadyan@iitg.ac.in, b.bikash@iitg.ac.in }

\date{}
\begin{document}
\maketitle

\vspace{-0.3in}

\begin{center}{\textbf{Abstract}}\end{center}

\noindent A mixed graph is called \emph{second kind hermitian integral}(or \emph{HS-integral}) if the eigenvalues of its Hermitian-adjacency matrix of second kind are integers. A mixed graph is called \emph{Eisenstein integral} if the eigenvalues of its (0, 1)-adjacency matrix are Eisenstein integers. Let $\Gamma$ be an abelian group. We characterize the set $S$ for which a mixed Cayley graph $\text{Cay}(\Gamma, S)$ is HS-integral. We also show that a mixed Cayley graph is Eisenstein integral if and only if it is HS-integral.

\vspace*{0.3cm}

\section{Introduction}
A \emph{mixed graph} $G$ is a pair $(V(G),E(G))$, where $V(G)$ and $E(G)$ are the vertex set and the edge set of $G$, respectively. Here $E(G)\subseteq V(G) \times V(G)\setminus \{(u,u)~|~u\in V(G)\}$. If $G$ is a mixed graph, then $(u,v)\in E(G)$ need not imply that $(v,u)\in E(G)$. An edge $(u,v)$ of a mixed graph $G$ is called \textit{undirected} if both $(u,v)$ and $(v,u)$ belong to $E(G)$. An edge $(u,v)$ of a mixed graph $G$ is called \textit{directed} if $(u,v)\in E(G)$ but $(v,u)\notin E(G)$. A mixed graph can have both undirected and directed edges. A mixed graph $G$ is said to be a \textit{simple graph} if all the edges of $G$ are undirected. A mixed graph $G$ is said to be an \textit{oriented graph} if all the edges of $G$ are directed. 

For a mixed graph $G$ on $n$ vertices, its (0, 1)-\textit{adjacency matrix} and \textit{Hermitian-adjacency matrix of second kind} are denoted by $\mathcal{A}(G)=(a_{uv})_{n\times n}$ and $\mathcal{H}(G)=(h_{uv})_{n\times n}$, respectively, where 

\[a_{uv} = \left\{ \begin{array}{rl}
	1 &\mbox{ if }
	(u,v)\in E \\ 
	0 &\textnormal{ otherwise,}
\end{array}\right.     ~~~~~\text{ and }~~~~~~ h_{uv} = \left\{ \begin{array}{cl}
	1 &\mbox{ if }
	(u,v)\in E \textnormal{ and } (v,u)\in E \\ \frac{1+i\sqrt{3}}{2} & \mbox{ if } (u,v)\in E \textnormal{ and } (v,u)\not\in E \\
	\frac{1-i\sqrt{3}}{2} & \mbox{ if } (u,v)\not\in E \textnormal{ and } (v,u)\in E\\
	0 &\textnormal{ otherwise.}
\end{array}\right.\] 

The Hermitian-adjacency matrix of second kind was introduced by Bojan Mohar~\cite{mohar2020new}. Let $G$ be a mixed graph. By an \emph{HS-eigenvalue} of $G$, we mean an eigenvalue of $\mathcal{H}(G)$. By an \emph{eigenvalue} of $G$, we mean an eigenvalue of $\mathcal{A}(G)$. Similarly, the \emph{HS-spectrum} of $G$, denoted $Sp_H(G)$, is the multi-set of the HS-eigenvalues of $G$, and the \emph{spectrum} of $G$, denoted $Sp(G)$, is the multi-set of the eigenvalues of $G$. Note that the Hermitian-adjacency matrix of second kind of a mixed graph is a Hermitian matrix, and so its HS-eigenvalues are real numbers. However, if a mixed graph $G$ contains at least one directed edge, then $\mathcal{A}(G)$ is non-symmetric. Accordingly, the eigenvalues of $G$ need not be real numbers. The matrix obtained by replacing $\frac{1+i\sqrt{3}}{2}$ and $\frac{1-i\sqrt{3}}{2}$ by $i$ and $-i$, respectively, in $\mathcal{H}(G)$, is called the \emph{Hermitian adjacency} matrix of $G$. Hermitian adjacency matrix of mixed graphs was introduced in~\cite{2017mixed, 2015mixed}. 

A mixed graph is called \textit{H-integral} if the eigenvalues of its Hermitian adjacency matrix are integers. A mixed graph $G$ is said to be \textit{HS-integral} if all the HS-eigenvalues of $G$ are integers. A mixed graph $G$ is said to be \textit{Eisenstein integral} if all the eigenvalues of $G$ are Eisenstein integers. Note that complex numbers of the form $a+b\omega_3$, where $a,b\in \mathbb{Z}, \omega_3=\frac{-1+i\sqrt{3}}{2}$, are called \emph{Eisenstein} integers. An HS-integral simple graph is called an \emph{integral} graph. Note that $\mathcal{A}(G)=\mathcal{H}(G)$ for a simple graph $G$. Therefore in case of a simple graph $G$, the terms HS-eigenvalue, HS-spectrum and HS-integrality of $G$ are the same with that of eigenvalue, spectrum and integrality of $G$, respectively. 

Integrality of simple graphs have been extensively studied in the past. Integral graphs were first defined by Harary and Schwenk~\cite{harary1974graphs} in 1974 and proposed a classification of integral graphs. See \cite{balinska2002survey} for a survey on integral graphs. Watanabe and Schwenk \cite{watanabe1979note,watanabe1979integral} proved several interesting results on integral trees in 1979. Csikvari \cite{csikvari2010integral} constructed integral trees with arbitrary large diameters in 2010. Further research on integral trees can be found in \cite{brouwer2008integral,brouwer2008small,wang2000some, wang2002integral}. In $2009$, Ahmadi et al. \cite{ahmadi2009graphs} proved that only a fraction of $2^{-\Omega (n)}$ of the graphs on $n$ vertices have an integral spectrum. Bussemaker et al. \cite{bussemaker1976there} proved that there are exactly $13$ connected cubic integral graphs. Stevanovi{\'c} \cite{stevanovic20034} studied the $4$-regular integral graphs avoiding $\pm3$ in the spectrum, and Lepovi{\'c} et al. \cite{lepovic2005there} proved that there are $93$ non-regular, bipartite integral graphs with maximum degree four. 

Let $S$ be a subset, not containing the identity element, of a group $\Gamma$. The set $S$ is said to be \textit{symmetric} (resp. \textit{skew-symmetric}) if $S$ is closed under inverse (resp. $a^{-1} \not\in S$ for all $a\in S$). Define $\overline{S}= \{u\in S: u^{-1}\not\in S \}$. Clearly, $S\setminus \overline{S}$ is symmetric and $\overline{S}$ is skew-symmetric. The \textit{mixed Cayley graph} $G=\text{Cay}(\Gamma,S)$ is a mixed graph, where $V(G)=\Gamma$ and $E(G)=\{ (a,b): a^{-1}b\in S , a,b\in \Gamma\}$. If $S$ is symmetric then $G$ is a \textit{simple Cayley graph}. If $S$ is skew-symmetric then $G$ is an \textit{oriented Cayley graph}.

	In 1982, Bridge and Mena \cite{bridges1982rational} introduced a characterization of integral Cayley graphs over abelian groups. Later on, same characterization was rediscovered by Wasin So \cite{2006integral} for cyclic groups in 2005. In 2009, Abdollahi and Vatandoost \cite{abdollahi2009cayley} proved that there are exactly seven connected cubic integral Cayley graphs. On the same year, Klotz and Sander \cite{klotz2010integral} proved that if a Cayley graph $\text{Cay}(\Gamma,S)$ over an abelian group $\Gamma$ is integral then $S$ belongs to the Boolean algebra $\mathbb{B}(\Gamma)$ generated by the subgroups of $\Gamma$. Moreover, they conjectured that the converse is also true, which was proved by Alperin and Peterson \cite{alperin2012integral}. In 2015, Ku et al. \cite{ku2015Cayley} proved that normal Cayley graphs over the symmetric groups are integral. In 2017, Lu et al. \cite{lu2018integral} gave necessary and sufficient condition for the integrality of Cayley graphs over the dihedral group $D_n$. In particular, they completely determined all integral Cayley graphs over the dihedral group $D_p$ for a prime $p$. In 2019, Cheng et al. \cite{cheng2019integral} obtained several simple sufficient conditions for the integrality of Cayley graphs over the dicyclic group $T_{4n}= \langle a,b| a^{2n}=1, a^n=b^2,b^{-1}ab=a^{-1} \rangle $. In particular, they also completely determined all integral Cayley graphs over the dicyclic group $T_{4p}$ for a prime $p$. In 2014, Godsil \emph{et al.} \cite{godsil2014rationality} characterized integral normal Cayley graphs. Xu \emph{et al.} \cite{xu2011gaussian} and Li \cite{li2013circulant} characterized the set $S$ for which the mixed circulant graph $\text{Cay}(\mathbb{Z}_n, S)$ is Gaussian integral. In 2006, So \cite{2006integral} introduced characterization of integral circulant graphs. In \cite{kadyan2021integralNormal}, the authors provide an alternative proof of the characterization obtained in~\cite{li2013circulant,xu2011gaussian}. H-integral mixed circulant graphs, H-integral mixed Cayley graphs over abelian groups, H-integral normal Cayley graphs and HS-integral mixed circulant graphs have been characterized in \cite{kadyan2021integral}, \cite{kadyan2021integralAbelian}, \cite{kadyan2021integralNormal} and \cite{kadyan2021Secintegral}, respectively.     

Throughout this paper, we consider Cayley graphs over abelian groups. The paper is organized as follows. In Section~\ref{prelAbelian}, some preliminary concepts and results are discussed. In particular, we express the HS-eigenvalues of a mixed Cayley graph as a sum of HS-eigenvalues of a simple Cayley graph and an oriented Cayley graph. In Section~\ref{sec3}, we obtain a sufficient condition on the connection set for the HS-integrality of an oriented Cayley graph. In Section~\ref{sec4}, we first characterize HS-integrality of oriented Cayley graphs by proving the necessity of the condition obtained in Section~\ref{sec3}. After that, we extend this characterization to mixed Cayley graphs. In Section~\ref{sec5}, we prove that a mixed Cayley graph is Eisenstein integral if and only if it is HS-integral.


\section{Preliminaries}\label{prelAbelian}

A \textit{representation} of a finite group $\Gamma$ is a homomorphism $\rho : \Gamma \to GL(V)$, where $GL(V)$ is  the group of automorphisms of a finite dimensional vector space $V$ over the complex field $\mathbb{C}$. The dimension of $V$ is called the \textit{degree} of $\rho$. Two representations $\rho_1$ and $\rho_2$ of $\Gamma$ on $V_1$ and $V_2$, respectively, are \textit{equivalent} if there is an isomorphism $T:V_1 \to V_2$ such that $T\rho_1(g)=\rho_2(g)T$ for all $g\in \Gamma$.

Let $\rho : \Gamma \to GL(V)$ be a representation. The \textit{character} $\chi_{\rho}: \Gamma \to \mathbb{C}$ of $\rho$ is defined by setting $\chi_{\rho}(g)=Tr(\rho(g))$ for $g\in \Gamma$, where $Tr(\rho(g))$ is the trace of the representation matrix of $\rho(g)$. By degree of $\chi_{\rho}$ we mean the degree of $\rho$ which is simply $\chi_{\rho}(1)$. If $W$ is a $\rho(g)$-invariant subspace of $V$ for each $g\in \Gamma$, then we say $W$ a $\rho(\Gamma)$-invariant subspace of $V$. If the only $\rho(\Gamma)$-invariant subspaces of $V$ are $\{ 0\}$ and $V$,  we say $\rho$ an \textit{irreducible representation} of $\Gamma$, and the corresponding character $\chi_{\rho}$ an \textit{irreducible character} of $\Gamma$. 

For a group $\Gamma$, we denote by $\text{IRR}(\Gamma)$ and $\text{Irr}(\Gamma)$ the complete set of non-equivalent irreducible representations of $\Gamma$  and the complete set of non-equivalent irreducible characters of $\Gamma$, respectively.

Throughout this paper, we consider $\Gamma$ to be an abelian group of order $n$. Let $S$ be a subset of $\Gamma$ with $0\not\in S$, where $0$ is the additive identity of $\Gamma$.  Then $\Gamma$ is isomorphic to the direct product of cyclic groups of prime power order, $i.e.$ 
 $$\Gamma\cong \mathbb{Z}_{n_1} \otimes \cdots \otimes \mathbb{Z}_{n_k},$$
 where $n=n_1 \cdots n_k$, and $n_j$ is a power of a prime number for each $j=1,...,k  $. We consider an abelian group $\Gamma$ as $\mathbb{Z}_{n_1} \otimes \cdots \otimes \mathbb{Z}_{n_k}$ of order $n=n_1...n_k$. We consider the elements $x\in \Gamma $ as elements of the cartesian product $\mathbb{Z}_{n_1} \otimes \cdots \otimes \mathbb{Z}_{n_k}$, $i.e.$ 
 $$x=(x_1,x_2,...,x_k),  \mbox{ where } x_j \in \mathbb{Z}_{n_j} \mbox{ for all } 1\leq j \leq k. $$
Addition in $\Gamma$ is done coordinate-wise modulo $n_j$. For a positive integer $k$ and $a\in \Gamma$, we denote by $ka$ or $a^k$ the $k$-fold sum of $a$ to itself, $(-k)a=k(-a)$, $0a=0$, and inverse of $a$ by $-a$. 

\begin{lema}\label{lemma1}\cite{steinberg2009representation}
Let $\mathbb{Z}_n=\{ 0,1,...,n-1\}$ be a cyclic group of order $n$. Then $\text{IRR}(\mathbb{Z}_n)=\{ \phi_k: 0\leq k \leq n-1\}$, where $\phi_k(j)=\omega_n^{jk}$ for all $0\leq j,k \leq n-1$, and $\omega_n=\exp(\frac{2\pi i}{n})$.
\end{lema}

\begin{lema}\label{lemma2}\cite{steinberg2009representation}
Let $\Gamma_1$,$\Gamma_2$ be abelian groups of order $m,n$, respectively. Let $\text{IRR}(\Gamma_1)=\{ \phi_1,...,\phi_m\}$, and $\text{IRR}(\Gamma_2)=\{ \rho_1,...,\rho_n\}$. Then $$\text{IRR}(\Gamma_1 \times \Gamma_2)=\{ \psi_{kl} : 1\leq k \leq m, 1\leq l \leq n \},$$ 
where $\psi_{kl}: \Gamma_1 \times \Gamma_2 \to \mathbb{C}^* \mbox{ and } \psi_{kl}(g_1,g_2)=\phi_k(g_1)\rho_l(g_2)$ for all $g_1\in \Gamma_1, g_2\in \Gamma_2$.
\end{lema}

Consider $\Gamma = \mathbb{Z}_{n_1}\times \mathbb{Z}_{n_2}\times ...\times  \mathbb{Z}_{n_k}$. By Lemma \ref{lemma1} and Lemma \ref{lemma2}, $\text{IRR}(\Gamma)=\{ \psi_{\alpha}: \alpha \in \Gamma\}$, where 
\begin{equation}
\psi_{\alpha}(x)=\prod_{j=1}^{k}\omega_{n_j}^{\alpha_j x_j} \textnormal{ for all $\alpha=( \alpha_1,...,\alpha_k),x=(x_1,...,x_k) \in \Gamma$},\label{character}
\end{equation}
and $\omega_{n_j}=\exp\left(\frac{2\pi i}{n_j}\right)$. Since $\Gamma$ is an abelian group, every irreducible representation of $\Gamma$ is 1-dimensional and thus it can be identified with its characters. Hence $\text{IRR}(\Gamma)=\text{Irr}(\Gamma)$. For $x\in \Gamma$, let $\text{ord}(x)$ denote the order of $x$. The following lemma can be easily proved.

\begin{lema}\label{Basic}
Let $\Gamma$ be an abelian group and $\text{Irr}(\Gamma)=\{\psi_{\alpha} : \alpha \in \Gamma \}$. Then the following statements are true.
\begin{enumerate}[label=(\roman*)]
\item $\psi_{\alpha}(x)=\psi_x({\alpha})$ for all $x,\alpha \in \Gamma$.
\item $(\psi_{\alpha}(x))^{\text{ord}(x)}=(\psi_{\alpha}(x))^{\text{ord}(\alpha)}=1$ for all $x,\alpha \in \Gamma$.
\end{enumerate} 
\end{lema}

Let $f : \Gamma \to \mathbb{C}$ be a function. The \textit{Cayley color digraph} of $\Gamma$ with \textit{connection function} $f$, denoted by $\text{Cay}(\Gamma, f)$, is defined to be the directed graph with vertex set $\Gamma$ and arc set $\{ (x,y): x,y \in \Gamma\}$ such that each arc $(x,y)$ is colored by $f(x^{-1}y)$. The \textit{adjacency matrix} of $\text{Cay}(\Gamma, f)$ is defined to be the matrix whose rows and columns are indexed by the elements of $\Gamma$, and the $(x,y)$-entry  is equal to $f(x^{-1}y)$. The eigenvalues of $\text{Cay}( \Gamma, f)$ are simply the eigenvalues of its adjacency matrix.

\begin{theorem}\cite{babai1979spectra}\label{EigNorColCayMix}
Let $\Gamma$ be a finite abelian group and $\text{Irr}(\Gamma)=\{\psi_{\alpha} : \alpha \in \Gamma \}$. Then the spectrum of the Cayley color digraph $\text{Cay}(\Gamma, f)$ is $\{ \gamma_\alpha : \alpha\in \Gamma \},$ where $$\gamma_{\alpha} = \sum_{y\in \Gamma} f(y)\psi_{\alpha}(y) \hspace{0.2cm} \textnormal{ for all } \alpha \in \Gamma.$$
\end{theorem} 

For a subset $S$ of an abelian group $\Gamma$, let $S^{-1}=\{s^{-1}~:~s\in S\}$.

\begin{lema}\cite{babai1979spectra}\label{imcgoa3}
Let $\Gamma$ be an abelian group and $\text{Irr}(\Gamma)=\{\psi_{\alpha} : \alpha \in \Gamma \}$. Then the HS-spectrum of the mixed Cayley graph $\text{Cay}(\Gamma, S)$ is $\{ \gamma_\alpha : \alpha \in \Gamma \}$, where  $\gamma_{\alpha}=\lambda_{\alpha}+\mu_{\alpha}$ and $$\lambda_{\alpha}=\sum_{s\in S\setminus \overline{S}} \psi_{\alpha}(s),\hspace{0.2cm} \mu_{\alpha}=\sum_{s\in\overline{S}}\left( \omega_6 \psi_{\alpha}(s)+ \omega_6^5\psi_{\alpha}(-s)\right) \textnormal{ for all } \alpha \in \Gamma.$$
\end{lema}
\begin{proof}
Define $f_S: \Gamma \to \{0,1,\omega_6,\omega_6^5\}$ such that $$f_S(s)= \left\{ \begin{array}{rl}
		1 & \mbox{if } s\in S \setminus \overline{S} \\
		\omega_6 & \mbox{if } s\in \overline{S}\\
		\omega_6^5 & \mbox{if } s\in \overline{S}^{-1}\\ 
		0 &   \mbox{otherwise}. 
	\end{array}\right.$$ 
The adjacency matrix of the Cayley color digraph $\text{Cay}(\Gamma, f_S)$ agrees with the Hermitian adjacency matrix of the mixed Cayley graph $\text{Cay}(\Gamma, S)$. Thus the result follows from Theorem~\ref{EigNorColCayMix}.
\end{proof}

Next two corollaries are special cases of Lemma~\ref{imcgoa3}.

\begin{corollary}\cite{klotz2010integral}\label{simpleAbelianEigBabai}
Let $\Gamma$ be an abelian group and $\text{Irr}(\Gamma)=\{\psi_{\alpha} : \alpha \in \Gamma \}$. Then the spectrum of the Cayley graph $\text{Cay}(\Gamma, S)$ is $\{ \lambda_\alpha : \alpha \in \Gamma \}$, where $\lambda_\alpha=\lambda_{-\alpha}$ and $$\lambda_{\alpha}=\sum_{s\in S} \psi_{\alpha}(s) \mbox{ for all } \alpha \in \Gamma.$$
\end{corollary}

\begin{corollary}\label{OriEig}
Let $\Gamma$ be an abelian group and $\text{Irr}(\Gamma)=\{\psi_{\alpha} : \alpha \in \Gamma \}$. Then the spectrum of the oriented Cayley graph $\text{Cay}(\Gamma, S)$ is $\{ \mu_\alpha : \alpha \in \Gamma \}$, where $$\mu_{\alpha}=\sum_{s\in S}\left(\omega_6 \psi_{\alpha}(s)+ \omega_6^5 \psi_{\alpha}(-s)\right) \mbox{ for all } \alpha \in \Gamma.$$
\end{corollary}

Let $n\geq 2$ be a fixed positive integer.  Define $G_n(d)=\{k: 1\leq k\leq n-1, \gcd(k,n)=d \}$. It is clear that $G_n(d)=dG_{\frac{n}{d}}(1)$. Alperin and Peterson \cite{alperin2012integral} considered a Boolean algebra generated by a class of subgroups of a group in order to determine the integrality of Cayley graphs over abelian groups. Suppose $\Gamma$ is a finite group, and $\mathcal{F}_{\Gamma}$ is the family of all subgroups of $\Gamma$. The Boolean algebra $\mathbb{B}(\Gamma)$ generated by $\mathcal{F}_{\Gamma}$ is the set whose elements are obtained by arbitrary finite intersections, unions, and complements of the elements in the family $\mathcal{F}_{\Gamma}$. The minimal non-empty elements of this algebra are called \textit{atoms}. Thus each element of $\mathbb{B}(\Gamma)$ is the union of some atoms.
Consider the equivalence relation $\sim$ on $\Gamma$ such that $x\sim y$ if and only if $y=x^k$ for some $k\in G_m(1)$, where $m=\text{ord}(x)$.

\begin{lema}\cite{alperin2012integral}  The equivalence classes of $\sim$ are the atoms of $\mathbb{B}(\Gamma)$.
\end{lema}

For $x\in \Gamma$, let $[x]$ denote the equivalence class of $x$ with respect to the relation $\sim$. Also, let $\langle x \rangle$ denote the cyclic group generated by $x$.

\begin{lema}\label{atomsboolean} \cite{alperin2012integral}  The atoms of the Boolean algebra $\mathbb{B}(\Gamma)$ are the sets $[x]=\{ y: \langle y \rangle = \langle x \rangle \}$.
\end{lema}

By Lemma \ref{atomsboolean}, each element of $\mathbb{B}(\Gamma)$ is a union of some sets of the form $[x]=\{ y: \langle y \rangle = \langle x \rangle \}$. Thus, for all $S\in \mathbb{B}(\Gamma)$, we have $S=[x_1]\cup...\cup [x_k]$ for some $x_1,...,x_k\in \Gamma$. The next result provides a complete characterization of integral Cayley graphs over an abelian group $\Gamma$ in terms of the atoms of $\mathbb{B}(\Gamma)$. 

\begin{theorem}\label{Cayint} (\cite{alperin2012integral}, \cite{bridges1982rational})
Let $\Gamma$ be an abelian group. The Cayley graph $\text{Cay}(\Gamma, S)$ is integral if and only if $S\in \mathbb{B}(\Gamma)$.
\end{theorem}

Define $\Gamma(3)$ to be the set of all $x\in \Gamma$ satisfying $\text{ord}(x)\equiv 0 \pmod 3$. For all $x\in \Gamma(3)$ and $r\in \{0,1,2\}$, define $$M_r(x):=\{x^k: 1\leq k \leq \text{ord}(x) , k \equiv r \Mod 3  \}.$$
 For all $a\in \Gamma$ and $S\subseteq \Gamma$, define $a+S:= \{ a+s: s\in S\}$ and $-S:=\{ -s: s\in S \}$. Note that $-s$ denotes the inverse of $s$, that is $-s=s^{m-1}$, where $m=\text{ord}(s)$.

\begin{lema} Let $\Gamma$ be an abelian group and $x\in \Gamma(3)$. Then the following statement\textbf{}s are true.
\begin{enumerate}[label=(\roman*)]
\item $\bigcup\limits_{r=0}^{2}M_r(x)= \langle x \rangle$.
\item Both $M_1(x)$ and $M_2(x)$ are skew-symmetric subsets of $\Gamma$.
\item $-M_1(x)=M_2(x)$ and $-M_2(x)=M_1(x)$.
\item $a+M_1(x)=M_1(x)$ and $a+M_2(x)=M_2(x)$ for all $a\in M_0(x)$.
\end{enumerate} 
\end{lema}
\begin{proof}
\begin{enumerate}[label=(\roman*)]
	\item It follows from the definitions of $M_r(x)$ and $\langle x \rangle$.
	\item Let $\text{ord}(x)=m$. If $x^k\in M_1(x)$ then $-x^k=x^{m-k} \not\in M_1(x)$, as $k\equiv 1 \pmod 3$ gives $m-k\equiv 2 \pmod 3$. Thus  $M_1(x)$ is a skew-symmetric subset of $\Gamma$. Similarly,  $M_2(x)$ is also a skew-symmetric subset of $\Gamma$.
	\item Let $\text{ord}(x)=m$. As $k\equiv 1 \pmod 3$ if and only if $m-k\equiv 2 \pmod 3$, and $-x^k=x^{m-k}$, we get $-M_1(x)=M_2(x)$ and $-M_2(x)=M_1(x)$.
	\item Let $a\in M_0(x)$ and  $y\in a+M_1(x)$. Then $a=x^{k_1}$ and $y=x^{k_1}+x^{k_2}=x^{k_1+k_2}$, where $k_1\equiv 0 \pmod 3$ and $k_2 \equiv 1 \pmod 3$. Since $k_1+k_2\equiv 1 \pmod 3$, we have $y\in M_1(x)$ implying that $a+M_1(x)\subseteq M_1(x)$. Hence $a+M_1(x)=M_1(x)$. Similarly, $a+M_2(x)=M_2(x)$ for all  $a\in M_0(x)$.\qedhere
\end{enumerate} 	
\end{proof}

Let $m \equiv 0 \Mod 3$. For $r\in \{1,2\}$ and $g \in \mathbb{Z}$, define the following:
\begin{align*}
&  G_{m,3}^r(1)=\{ k: 1\leq k \leq m-1 , \gcd(k,m )= 1,k\equiv r \Mod 3 \},\\
&  D_{g,3}= \{ k: k \text{ divides } g, k \not\equiv 0 \Mod 3\},~\text{ and}\\
&  D_{g,3}^r=\{ k: k \text{ divides } g, k \equiv r \Mod 3\} .
\end{align*}
It is clear that $D_{g,3}=D_{g,3}^1 \bigcup\limits D_{g,3}^2$. Define an equivalence relation $\approx$ on $\Gamma(3)$ such that $x \approx y$ if and only if $y=x^k$ for some  $k\in G_{m,3}^1(1)$, where $m=\text{ord}(x)$. Observe that if  $x,y\in \Gamma(3)$ and $x \approx y$ then $x \sim y$, but the converse need not be true. For example, consider $x=5\pmod {12}$, $y=7\pmod {12}$ in $\mathbb{Z}_{12}$. Here $x,y\in \mathbb{Z}_{12}(3)$ and $x \sim y$ but $x \not\approx y$. For $x\in \Gamma(3)$, let $\langle\!\langle x \rangle\!\rangle$ denote the equivalence class of $x$ with respect to the relation $\approx$.

\begin{lema}\label{lemanecc} Let $\Gamma$ be an abelian group, $x\in \Gamma(3)$ and $m=\text{ord}(x)$. Then the following are true.
\begin{enumerate}[label=(\roman*)]
\item $\langle\!\langle x \rangle\!\rangle= \{ x^k:  k \in G_{m,3}^1(1)  \}$.
\item $\langle\!\langle -x \rangle\!\rangle= \{ x^k: k \in G_{m,3}^2(1)  \}$.
\item $\langle\!\langle x \rangle\!\rangle \cap \langle\!\langle -x \rangle\!\rangle=\emptyset$.
\item $[x]=\langle\!\langle x \rangle\!\rangle \cup \langle\!\langle -x \rangle\!\rangle$.
\end{enumerate} 
\end{lema}
\begin{proof}
\begin{enumerate}[label=(\roman*)]
\item Let $y\in \langle\!\langle x \rangle\!\rangle$. Then $x \approx y$, and so $\text{ord}(x)=\text{ord}(y)=m$ and there exists $k\in G_{m,3}^1(x)$ such that $y=x^k$. Thus $\langle\!\langle x \rangle\!\rangle\subseteq \{ x^k:  k \in G_{m,3}^1(1)  \}$. On the other hand, let $z=x^k$ for some $k\in G_{m,3}^1(1)$. Then $\text{ord}(x)=\text{ord}(z)$ and so $x \approx z$. Thus $ \{ x^k:  k \in G_{m,3}^1(1)  \} \subseteq \langle\!\langle x \rangle\!\rangle$.
\item Note that $-x=x^{m-1}$ and $m-1\equiv 2 \pmod 3$. By Part $(i)$, 
\begin{align*}
\langle\!\langle -x \rangle\!\rangle = \{ (-x)^k:  k \in G_{m,3}^1(1)  \} = \{ x^{(m-1)k}:  k \in G_{m,3}^1(1)  \}&= \{ x^{-k}:  k \in G_{m,3}^1(1)  \}\\
&= \{ x^{k}:  k \in G_{m,3}^2(1)  \}.
\end{align*}
\item Since $G_{m,3}^1(1)\cap G_{m,3}^2(1)=\emptyset$, so by Part $(i)$ and Part $(ii)$, $\langle\!\langle x \rangle\!\rangle \cap \langle\!\langle -x \rangle\!\rangle=\emptyset$ holds.
\item  Since $[x]= \{ x^k:  k \in G_m(1)  \}$ and $G_m(1)$ is a disjoint union of $G_{m,3}^1(1)$ and $ G_{m,3}^3(1)$, by Part $(i)$ and Part $(ii)$, $[x]=\langle\!\langle x \rangle\!\rangle \cup \langle\!\langle -x \rangle\!\rangle$ holds. \qedhere
\end{enumerate} 
\end{proof}

\begin{lema}\label{imcgoa4} Let $\Gamma$ be an abelian group, $x\in \Gamma(3)$, $m=\text{ord}(x)$ and $g=\frac{m}{3}$. Then the following are true.
\begin{enumerate}[label=(\roman*)]
\item $M_1(x) \cup M_2(x)=\bigcup\limits_{h\in D_{g,3}} [x^h] $.
\item $M_1(x)= \bigcup\limits_{h\in D_{g,3}^1} \langle\!\langle x^h \rangle\!\rangle \cup \bigcup\limits_{h\in D_{g,3}^2} \langle\!\langle -x^h \rangle\!\rangle$.
\item $M_2(x)=\bigcup\limits_{h\in D_{g,3}^1} \langle\!\langle -x^h \rangle\!\rangle \cup \bigcup\limits_{h\in D_{g,3}^2} \langle\!\langle x^h \rangle\!\rangle $.
\end{enumerate}
\end{lema}
\begin{proof}
\begin{enumerate}[label=(\roman*)]
\item Let $x^k \in M_1(x) \cup M_2(x)$, where $k \equiv 1 \text{ or } 2 \pmod 3$. To show that $x^k \in \bigcup\limits_{h\in D_{g,3}} [x^h]$, it is enough to show $x^k \sim x^h$ for some $h \in D_{g,3}$. Let $h=\gcd (k,g) \in D_{g,3}$. Note that 
$$\text{ord}(x^k)=\frac{m}{\gcd(m,k)}=\frac{m}{\gcd(g,k)}=\frac{m}{h}=\text{ord}(x^h).$$ 
Also, as $h=\gcd(k,m)$, we have $\langle x^k \rangle = \langle x^h \rangle$, and so $x^k=x^{hj}$ for some $j\in G_q(1)$, where $q=\text{ord}(x^h)=\frac{m}{h}$. Thus $x^k\sim x^h$ where $h=\gcd (k,g) \in D_{g,3}$. 
Conversely, let $z\in \bigcup\limits_{h\in D_{g,3}} [x^h]$. Then there exists $h\in D_{g,3}$ such that $z=x^{hj}$ where $j\in G_q(1)$ and $q= \frac{m}{\gcd(m,h)}$. Now $h\in D_{g,3}$ and $q\equiv 0\pmod 3$ imply that $hj\equiv 1 \text{ or } 2 \pmod 3$, and so $\bigcup\limits_{h\in D_{g,3}} [x^h] \subseteq M_1(x) \cup M_2(x)$. Hence $M_1(x) \cup M_2(x)=\bigcup\limits_{h\in D_{g,3}} [x^h] $.
	
\item Let $x^k \in M_1(x)$, where $k\equiv 1 \pmod 3$. By Part $(i)$, there exists $h\in D_{g,3}$  and $j\in G_q(1)$ such that $x^k=x^{hj}$, where $q=\frac{m}{\gcd(m,h)}$. Note that $k=jh$. If $h\equiv 1 \pmod 3$ then $j\in G_{q,3}^1(1)$, otherwise $j\in G_{q,3}^2(1)$. Thus using parts $(i)$ and $(ii)$ of Lemma \ref{lemanecc}, if $h\equiv 1 \pmod 3$ then $x^k \approx x^h$, otherwise $x^k \approx -x^h$. Hence $M_1(x) \subseteq \bigcup\limits_{h\in D_{g,3}^1} \langle\!\langle x^h \rangle\!\rangle \cup \bigcup\limits_{h\in D_{g,3}^2} \langle\!\langle -x^h \rangle\!\rangle$. Conversely, assume that $z\in  \bigcup\limits_{h\in D_{g,3}^1} \langle\!\langle x^h \rangle\!\rangle \cup \bigcup\limits_{h\in D_{g,3}^2} \langle\!\langle -x^h \rangle\!\rangle$. This gives $z\in \langle\!\langle x^h \rangle\!\rangle$ for an $h\in D_{g,3}^1$ or $z\in \langle\!\langle -x^h \rangle\!\rangle$ for an $h\in D_{g,3}^2$. In the first case, by part $(i)$ of Lemma \ref{lemanecc},  there exists  $j\in G_{q,3}^1(1)$ with $q=\frac{m}{\gcd(m,h)}$ such that  $z = x^{hj}$. Similarly, for the second case, by part $(ii)$ of Lemma \ref{lemanecc},  there exists  $j\in G_{q,3}^2(1)$ with $q=\frac{m}{\gcd(m,h)}$ such that  $z = x^{hj}$. In both the cases, $hj \equiv 1 \pmod 3$. Thus $z\in M_1(x)$.

\item The proof is similar to Part $(ii)$. \qedhere
\end{enumerate}
\end{proof}

The \textit{cyclotomic polynomial} $\Phi_m(x)$ is the monic polynomial whose zeros are the primitive $m^{th}$ roots of unity. That is,
	$$\Phi_m(x)= \prod_{a\in G_m(1)}(x-\omega_m^a).$$ 
	Clearly, the degree of $\Phi_m(x)$ is $\varphi(m)$, where $\varphi$ denotes the Euler $\varphi$-function. It is well known that the cyclotomic polynomial $\Phi_m(x)$ is monic and irreducible in $\mathbb{Z}[x]$.  See \cite{numbertheory} for more details on cyclotomic polynomials.

	The polynomial $\Phi_m(x)$ is irreducible over $\mathbb{Q}(\omega_3)$ if and only if $[\mathbb{Q}(\omega_3,\omega_m) : \mathbb{Q}(\omega_3)]= \varphi(m)$. Also, $ \mathbb{Q}(\omega_m)$ does not contain the number $\omega_3$ if and only if $m\not\equiv 0 \Mod 3$. Thus, if $m\not\equiv 0 \Mod 3$ then $[\mathbb{Q}(\omega_3,\omega_m):\mathbb{Q}(\omega_m) ]=2=[\mathbb{Q}(\omega_3), \mathbb{Q}]$, and therefore 
	$$[\mathbb{Q}(\omega_3,\omega_m) : \mathbb{Q}(\omega_3)]=\frac{[\mathbb{Q}(\omega_3,\omega_m) : \mathbb{Q}(\omega_m)] \times [\mathbb{Q}(\omega_m) : \mathbb{Q}]}{ [\mathbb{Q}(\omega_3) : \mathbb{Q}]}= [\mathbb{Q}(\omega_m) : \mathbb{Q}]= \varphi(m).$$
Further, if $m\equiv 0 \Mod 3$ then	$ \mathbb{Q}(\omega_3,\omega_m)= \mathbb{Q}(\omega_m)$, and so 
$$[\mathbb{Q}(\omega_3,\omega_m) : \mathbb{Q}(\omega_3)] = \frac{[\mathbb{Q}(\omega_3,\omega_m) : \mathbb{Q}]}{[\mathbb{Q}(\omega_3) : \mathbb{Q}]}=\frac{\varphi(m)}{2}.$$
Note that $\mathbb{Q}(\omega_3)=\mathbb{Q}(\omega_6)=\mathbb{Q}(i\sqrt{3})$. Therefore $\Phi_m(x)$ is irreducible over $\mathbb{Q}(\omega_3), \mathbb{Q}(\omega_6)$ or $\mathbb{Q}(i\sqrt{3})$ if and only if $m\not\equiv 0 \Mod 3$.
	
Let $m\equiv 0 \Mod 3$. Observe that $G_{m}(1)$ is a disjoint union of $G_{m,3}^1(1)$ and $G_{m,3}^2(1)$. Define 
$$\Phi_{m,3}^{1}(x)= \prod_{a\in G_{m,3}^1(1)}(x-\omega_m^a)~~ \textnormal{ and } ~~\Phi_{m,3}^2(x)= \prod_{a\in G_{m,3}^2(1)}(x-\omega_m^a).$$ 
It is clear from the definition that $\Phi_m(x)=\Phi_{m,3}^1(x)\Phi_{m,3}^2(x)$.

\begin{theorem}\cite{kadyan2021Secintegral}
		Let $m\equiv 0\Mod 3$. Then $\Phi_{m,3}^1(x)$ and $\Phi_{m,3}^2(x)$ are irreducible monic polynomials in $\mathbb{Q}(\omega_3)[x]$ of degree $\frac{\varphi(m)}{2}$.
	\end{theorem}

\section{A sufficient condition for HS-integrality of oriented Cayley graphs over abelian groups}\label{sec3}

In this section, first we prove that $S=\emptyset$ is the only connection set for an HS-integral oriented Cayley graph $\text{Cay}(\Gamma, S)$ whenever $\Gamma(3)=\emptyset$. After that we obtain a sufficient condition on the set $S$ for which the oriented Cayley graph $\text{Cay}(\Gamma, S)$ is HS-integral.

\begin{lema}\label{AbelianSqrt3ZeroSum} Let $S$ be a skew-symmetric subset of an abelian group $\Gamma$. If $\sum\limits_{s \in S} i\sqrt{3} (\psi_{\alpha}(s) -\psi_{\alpha}(-s)) =0$  for all $j=0,...,n-1$ then $S=\emptyset$.
	\end{lema}
	\begin{proof} Let $A_S=(a_{uv})_{n\times n}$ be the matrix whose rows and columns are indexed by the elements of $\Gamma$, where 
	$$a_{uv} = \left\{ \begin{array}{cl}
		i\sqrt{3} & \mbox{ if } v-u \in S \\
		-i\sqrt{3} & \mbox{ if } v-u \in S^{-1}\\
		0 &\textnormal{ otherwise.}
	\end{array}\right.$$ 
Since $A_S$ is a circulant matrix,  $\lambda_{\alpha}=\sum\limits_{k\in S} i\sqrt{3} (\psi_{\alpha}(s)-\psi_{\alpha}(-s))$ is an HS-eigenvalue of $A_S$ for each $\alpha \in \Gamma$. Therefore $\lambda_{\alpha}=0$ for all $\alpha \in \Gamma$, which implies all the entries of $A_S$ are zero. Hence $S=\emptyset$.
	\end{proof}

\begin{theorem}\label{ori4}
Let $\Gamma$ be an abelian group and $\Gamma(3) = \emptyset$. Then the oriented Cayley graph $\text{Cay}(\Gamma, S)$ is HS-integral if and only if $S=\emptyset$
\end{theorem}
\begin{proof}  
Let $G=\text{Cay}(\Gamma, S)$ and $Sp_H(G)=\{ \mu_{\alpha}: \alpha \in \Gamma \}$. Assume that $\text{Cay}(\Gamma, S)$ is HS-integral and $n\not\equiv 0 \Mod 3$. By Corollary~\ref{OriEig}, $$\mu_{\alpha}=\sum_{s\in S} (\omega_6 \psi_{\alpha}(s)+ \omega_6^5\psi_{\alpha}(-s)) \in \mathbb{Z}, \textnormal{ for all } \alpha \in \Gamma.$$ 
Note that, $\psi_{\alpha}(s)$ and $\psi_{\alpha}(-s)$ are $n^{th}$ roots of unity for all $ \alpha \in \Gamma, s\in S$.
Fix a primitive $n^{th}$ root $\omega$ of unity and express $\psi_{\alpha}(s)$ in the form $\omega^j$ for some $j \in \{ 0,1,...,n-1\}$. Thus $$\mu_{\alpha}= \sum_{s\in S} (\omega_6 \psi_{\alpha}(s)+ \omega_6^5\psi_{\alpha}(-s) ) = \sum_{j=0}^{n-1} a_j \omega^j,$$
 where $a_j \in \mathbb{Q}(\omega_3)$. Since $\mu_{\alpha} \in \mathbb{Z}$, so $p(x)= \sum\limits_{j=0}^{n-1} a_j x^j- \mu_{\alpha} \in \mathbb{Q}(\omega_3)[x]$ and $\omega$ is a root of $p(x)$. Since $n\not\equiv 0(\mod 3)$, so $\Phi_n(x)$ is irreducible in $\mathbb{Q}(\omega_3)[x]$. Thus $p(\omega)=0$ and $\Phi_n(x)$ is the monic irreducible polynomial over $\mathbb{Q}(\omega_3)$ having $\omega$ as a root. Therefore $\Phi_n(x)$ divides $p(x)$, and so $\omega^{-1}=\omega^{n-1}$ is also a root of $p(x)$. Note that, if $\psi_{\alpha}(s)=\omega^j$ for some $j \in \{ 0,1,...,n-1\}$ then $\psi_{-\alpha}(s)=\omega^{-j}$. We have 
\begin{align*}
\sum_{s \in S} i\sqrt{3} (\psi_{\alpha}(s) -  \psi_{\alpha}(-s))&=\sum_{s \in S}[(\omega_6 - \omega_6^5)  \psi_{\alpha}(s) + (\omega_6^5 - \omega_6) \psi_{\alpha}(-s)]\\
&=\sum_{j=0}^{n-1} a_j \omega^{-j}- \mu_{\alpha}=\mu_{-\alpha}-\mu_{\alpha}=p(\omega^{-1})=0 .
\end{align*} 
By Lemma \ref{AbelianSqrt3ZeroSum}, $S=\emptyset$. Conversely, if $S=\emptyset$ then all the HS-eigenvalues of $\text{Cay}(\Gamma, S)$ are zero. Thus $\text{Cay}(\Gamma, S)$ is HS-integral.
\end{proof}

\begin{lema}\label{imcgoa11}
Let $\Gamma$ be an abelian group and $x\in \Gamma(3)$. Then $\sum\limits_{s\in M_1(x)} \left(\omega_6 \psi_{\alpha}(s) + \omega_6^5 \psi_{\alpha}(-s)\right)$ is an integer for each $\alpha \in \Gamma$.
\end{lema}
\begin{proof} Let $x\in \Gamma(3)$, $\alpha \in \Gamma$ and  $\mu_{\alpha}=\sum\limits_{s\in M_1(x)}\left( \omega_6 \psi_{\alpha}(s) + \omega_6^5 \psi_{\alpha}(-s)\right).$

\noindent\textbf{Case 1.} There exists $a\in M_0(x)$ such that $\psi_{\alpha}(a)\neq 1$. Then 

\begin{equation*}
\begin{split}
\mu_{\alpha}=\sum\limits_{s\in M_1(x)}\left( \omega_6 \psi_{\alpha}(s) + \omega_6^5 \psi_{\alpha}(-s)\right)
&= \sum\limits_{s\in M_1(x)} \omega_6 \psi_{\alpha}(s) + \sum\limits_{s\in M_2(x)} \omega_6^5 \psi_{\alpha}(s)\\
&= \sum\limits_{s\in a+M_1(x)} \omega_6 \psi_{\alpha}(s) + \sum\limits_{s\in a+M_2(x)} \omega_6^5 \psi_{\alpha}(s)\\
&=\psi_{\alpha}(a) \sum\limits_{s\in M_1(x)} \omega_6 \psi_{\alpha}(s) + \psi_{\alpha}(a) \sum\limits_{s\in M_2(x)} \omega_6^5 \psi_{\alpha}(s)\\
&= \psi_{\alpha}(a) \mu_{\alpha}.
\end{split}
\end{equation*} 
We have $(1-\psi_{\alpha}(a)) \mu_{\alpha}=0$. Since $\psi_{\alpha}(a)\neq 1$, we get $\mu_{\alpha}=0\in \mathbb{Z}$.

\noindent\textbf{Case 2.} Assume that $\psi_{\alpha}(a)=1$ for all $a\in M_0(x)$. Then $\psi_{\alpha}(s)=\psi_{\alpha}(x)$ for all $s\in M_1(x)$ and $\psi_{\alpha}(s)=\psi_{\alpha}(x^2)$ for all $s\in M_2(x)$. Therefore
\begin{equation*}
\begin{split}
\mu_{\alpha}=\sum\limits_{s\in M_1(x)}\left( \omega_6 \psi_{\alpha}(s) + \omega_6^5 \psi_{\alpha}(-s)\right)
&= \sum\limits_{s\in M_1(x)} \omega_6 \psi_{\alpha}(s) + \sum\limits_{s\in M_2(x)} \omega_6^5 \psi_{\alpha}(s)\\
&= |M_1(x)|(\omega_6 \psi_{\alpha}(x) + \omega_6^5 \psi_{\alpha}(x^2))\\
&= -|M_1(x)|(\omega_3^2 \psi_{\alpha}(x) + \omega_3 \psi_{\alpha}(x^2)).
\end{split}
\end{equation*}
Since $\psi_{\alpha}(x)^3=1$ then $\psi_{\alpha}(x)=\omega_3$ or $\omega_3^2$. If $\psi_{\alpha}(x)=\omega_3$ then $\mu_{\alpha}= -2 |M_1(x)|$. If $\psi_{\alpha}(x)=\omega_3^2$ then $\mu_{\alpha}=  |M_1(x)|$. Thus in both cases, $\mu_{\alpha}$ are integers for all $\alpha \in \Gamma$. 
\end{proof}

For $x\in \Gamma(3)$ and $\alpha \in \Gamma$, define 
$$Z_{x}(\alpha)= \sum\limits_{s\in \langle\!\langle x \rangle\!\rangle}\left( \omega_6 \psi_{\alpha}(s) + \omega_6^5 \psi_{\alpha}(-s)\right).$$

\begin{lema}\label{integerEigenvalue}
Let $\Gamma$ be an abelian group and $x\in \Gamma(3)$. Then $Z_{x}(\alpha)$ is an integer for each $\alpha \in \Gamma$.
\end{lema}
\begin{proof} Note that there exists $x\in \Gamma(3)$ with $\text{ord}(x)=3$. Apply induction on $\text{ord}(x)$. If $\text{ord}(x)=3$, then $M_1(x)=\langle\!\langle x \rangle\!\rangle$. Hence by Lemma ~\ref{imcgoa11}, $Z_{x}(\alpha)$ is an integer for each $\alpha \in \Gamma$. Assume that the statement  holds for all $x\in \Gamma(3)$ with $\text{ord}(x)\in \{ 3,6,...,3(g-1)\}$. We prove it for $\text{ord}(x)=3g$. Lemma ~\ref{imcgoa4} implies that
$$M_1(x)= \bigcup\limits_{h\in D_{g,3}^1} \langle\!\langle x^h \rangle\!\rangle \cup \bigcup\limits_{h\in D_{g,3}^2} \langle\!\langle -x^h \rangle\!\rangle.$$ 
If $\text{ord}(x)=3g=m, h\in D_{g,3}^1\cup D_{g,3}^2$, and $h>1$ then $\text{ord}(x^h), \text{ord}(-x^h)\in \{ 3,6,...,3(g-1)\}$. By induction hypothesis, both $Z_{x^h}(\alpha)$ and $Z_{-x^h}(\alpha)$ are integers for all $\alpha \in \Gamma$. Now we have 
\begin{equation*}
\begin{split}
\sum\limits_{s\in M_1(x)}\left( \omega_6\psi_{\alpha}(s)+ \omega_6^5\psi_{\alpha}(s)\right) &= Z_{x}(\alpha)+ \sum_{h\in D_{g,3}^1, h> 1} Z_{x^h}(\alpha) + \sum_{h\in D_{g,3}^2, h> 1} Z_{-x^h}(\alpha).
\end{split} 
\end{equation*}
By Lemma ~\ref{imcgoa11} and induction hypothesis,
\begin{equation*}
\begin{split}
Z_{x}(\alpha) = \sum\limits_{s\in M_1(x)}\left( \omega_6\psi_{\alpha}(s)+ \omega_6^5\psi_{\alpha}(s)\right) - \sum_{h\in D_{g,3}^1, h> 1} Z_{x^h}(\alpha) - \sum_{h\in D_{g,3}^2, h> 1} Z_{-x^h}(\alpha)
\end{split} 
\end{equation*} 
is an integer for each $\alpha \in \Gamma$.
\end{proof}

 For $\Gamma(3) \neq \emptyset$, define $\mathbb{E}(\Gamma)$ to be the set of all skew-symmetric subsets of $\Gamma$ of the form $\langle\!\langle x_1 \rangle\!\rangle\cup...\cup \langle\!\langle x_k \rangle\!\rangle$ for some $x_1,...,x_k\in \Gamma(3)$. For  $\Gamma(3) = \emptyset$, define $\mathbb{E}(\Gamma)=\{ \emptyset \}$.
 
\begin{theorem}\label{OrientedChara}
Let $\Gamma$ be an abelian group. If $S \in \mathbb{E}(\Gamma)$ then the oriented Cayley graph  $\text{Cay}(\Gamma, S)$ is HS-integral.
\end{theorem}
\begin{proof}
Assume that $S \in \mathbb{E}(\Gamma)$. Then $S=\langle\!\langle x_1 \rangle\!\rangle\cup...\cup \langle\!\langle x_k \rangle\!\rangle$ for some $x_1,...,x_k\in \Gamma(3)$. We have 	
\begin{equation*}
	\begin{split}
\mu_{\alpha} &=\sum_{s\in S}\left( \omega_6 \psi_{\alpha}(s)+ \omega_6^5 \psi_{\alpha}(-s)\right)= \sum_{j=1}^k Z_{x_j}(\alpha).
	\end{split} 
\end{equation*}
Now by Lemma~\ref{integerEigenvalue}, $\mu_{\alpha}$ is an integer for each $\alpha \in \Gamma$. Hence the oriented Cayley graph  $\text{Cay}(\Gamma, S)$ is HS-integral.
\end{proof}


\section{Characterization of HS-integral mixed Cayley graphs over abelian groups}\label{sec4}

Let $\Gamma$ be an abelian group of order $n$. Define $E$ to be the matrix of size $n\times n$,  whose rows and columns are indexed by elements of $\Gamma$ such that $E_{x,y}=\psi_{x}(y)$. Note that each row of $E$ corresponds to a character of $\Gamma$ and $EE^*=nI_n$, where $E^*$ is the conjugate transpose of $E$. Let $v_{\langle\!\langle x \rangle\!\rangle}$ be the vector in $\mathbb{Q}(\omega_3)^n$ whose coordinates are indexed by the  elements of $\Gamma$, and the $a^{th}$ coordinate of $v_{\langle\!\langle x \rangle\!\rangle}$ is given by
	$$v_{\langle\!\langle x \rangle\!\rangle}(a) = \left\{ \begin{array}{cl}
		\omega_6 &\mbox{ if }
		a \in \langle\!\langle x \rangle\!\rangle \\\omega_6^5 & \mbox{ if } a \in \langle\!\langle -x \rangle\!\rangle\\
		0 &\textnormal{ otherwise.}
	\end{array}\right.$$ By Lemma~\ref{integerEigenvalue}, we have $Ev_{\langle\!\langle x \rangle\!\rangle} \in \mathbb{Z}^n$. For $z\in \mathbb{C}$, let $\overline{z}$ denote the complex conjugate of $z$ and $\Re (z)$ (resp. $\Im (z)$) denote the real part (resp. imaginary part) of $z$.

\begin{lema}\label{Ori4Nec} Let $\Gamma$ be an abelian group, $v\in \mathbb{Q}(\omega_3)^n$ and $Ev \in \mathbb{Q}^n$. Let the coordinates of $v$ be indexed by elements of $\Gamma$. Then
\begin{enumerate}[label=(\roman*)]
\item $\overline{v}_x=v_{-x}$ for all $x \in \Gamma$.
\item $v_x=v_y$ for all $x,y \in \Gamma(3)$ satisfying $x \approx y$.
\item $\Re(v_x)=\Re(v_{-x})$ and $\Im(v_x)=\Im(v_{-x})=0$ for all $x\in \Gamma \setminus \Gamma(3)$.
\end{enumerate} 
\end{lema}
\begin{proof} Let $E_x$ and $E_y$ denote the column vectors of $E$ indexed by $x$ and $y$, respectively, and assume that $u=Ev \in \mathbb{Q}^n$. 
\begin{enumerate}[label=(\roman*)]
\item We use the fact that $\overline{\psi_x(y)}=\psi_{-x}(y)=\psi_x(-y)$ for all $x,y \in \Gamma$. Again
\[u=Ev\Rightarrow E^*u=E^*Ev= (nI_n)v\Rightarrow \frac{1}{n}E^*u=v \in \mathbb{Q}(\omega_3)^n.\]
Thus
 \begin{equation*}
	\begin{split}
v_x=\frac{1}{n}(E^* u)_x=\frac{1}{n} \sum_{a\in \Gamma} E^*_{x,a}u_a = \frac{1}{n} \sum_{a\in \Gamma} \overline{ \psi_a(x)}u_a &= \frac{1}{n} \sum_{a\in \Gamma}  \psi_{a}(-x)u_a \\
&=  \overline{\frac{1}{n} \sum_{a\in \Gamma} \overline{ \psi_{a}(-x)}u_a}\\
&= \overline{\frac{1}{n} \sum_{a\in \Gamma} E^*_{-x,a}u_a}=  \overline{\frac{1}{n}(E^* u)_{-x}}=\overline{v}_{-x}.
	\end{split} 
\end{equation*}

\item If $\Gamma(3) = \emptyset$ then there is nothing to prove. Now assume that $\Gamma(3)\neq\emptyset$.
Let $x,y \in \Gamma(3)$ and $x \approx y$. Then there exists $k\in G_{m,3}^1(1)$ such that $y=x^k$, where $m=\text{ord}(x)$. Assume $x\neq y$, so that $k\geq 2$.
Using Lemma ~\ref{Basic}, entries of $E$ are $m^{th}$ roots of unity. Fix a primitive $m^{th}$ root of unity $\omega$, and express each entry of $E_x$ and $E_y$ in the form $\omega^j$ for some $j\in \{ 0,1,...,m-1\}$. 
Thus $$nv_x= (E^*u)_x= \sum_{j=0}^{m-1} a_j \omega^j,$$ where $a_j\in \mathbb{Q}$ for all $j$. Thus $\omega$ is a root of the polynomial $p(x)= \sum\limits_{j=0}^{m-1} a_j x^j-nv_x \in \mathbb{Q}(\omega_3)[x]$. Therefore $p(x)$ is a multiple of the irreducible polynomial $\Phi_{m,3}^1(x)$, and so $\omega^k$ is also a root of $p(x)$, because of $k\in G_{m,3}^1(1)$.
As $y=x^k$ implies that $\psi_{a}(y)=\psi_{a}(x)^k$ for all $a\in \Gamma$, we have $(E^*u)_y= \sum\limits_{j=0}^{m-1} a_j \omega^{kj}$. Hence 
$$0 =p(\omega^k)=  \sum\limits_{j=0}^{m-1} a_j \omega^{kj}-nv_x= (E^*u)_y -nv_x=nv_y-nv_x \Rightarrow v_x=v_y.$$

\item Let $x\in \Gamma \setminus \Gamma(3)$ and $r=\text{ord}(x) \not\equiv 0 \pmod 3$. Fix a primitive $r^{th}$ root $\omega$ of unity, and express each entry of $E_x$ in the form $\omega^j$ for some $j\in \{ 0,1,...,r-1\}$. 
Thus $$nv_x= (E^*u)_x= \sum_{j=0}^{r-1} a_j \omega^j,$$ where $a_j\in \mathbb{Q}$ for all $j$. Thus $\omega$ is a root of the polynomial $p(x)= \sum\limits_{j=0}^{r-1} a_j x^j-nv_x \in \mathbb{Q}(\omega_3)[x]$. Therefore, $p(x)$ is a multiple of the irreducible polynomial $\Phi_r(x)$, and so $\omega^{-1}$ is also a root of $p(x)$. Since $\psi_{a}(-x)=\psi_{a}(x)^{-1}$ for all $a\in \Gamma$, therefore $(E^*u)_{-x}= \sum\limits_{j=0}^{r-1} a_j \omega^{-j}$. Hence $$0 =p(\omega^{-1})=  \sum\limits_{j=0}^{r-1} a_j \omega^{-j}-nv_x= (E^*u)_{-x} -nv_x=nv_{-x}-nv_x,$$ implies that $v_x=v_{-x}$. This together with Part $(i)$ imply that $\Re(v_x)=\Re(v_{-x})$, and that $\Im(v_x)=\Im(v_{-x})=0$ for all $x\in \Gamma \setminus \Gamma(3)$. \qedhere
\end{enumerate} 
\end{proof}

\begin{theorem}\label{neccori}
Let $\Gamma$ be an abelian group. The oriented Cayley graph  $\text{Cay}(\Gamma, S)$ is HS-integral if and only if $S \in \mathbb{E}(\Gamma)$.
\end{theorem}
\begin{proof}
Assume that the oriented Cayley graph  $\text{Cay}(\Gamma, S)$ is integral. If $\Gamma(3) = \emptyset$ then by Theorem ~\ref{ori4}, we have $S = \emptyset$, and so $S \in \mathbb{E}(\Gamma)$. Now assume that $\Gamma(3) \neq \emptyset$. Let $v$ be the vector in $\mathbb{Q}^n(\omega_3)$ whose coordinates are indexed by the elements of $\Gamma$, and the $x^{th}$ coordinate of $v$ is given by
	$$v_x = \left\{ \begin{array}{rl}
		\omega_6 &\mbox{ if }
		x \in S \\ \omega_6^5 & \mbox{ if } x \in S^{-1}\\
		0 &\textnormal{ otherwise.}
	\end{array}\right.$$
We have
\[(Ev)_a=\sum\limits_{x\in \Gamma}E_{a,x}v_x =\sum\limits_{x\in S}\omega_6 E_{a,x}+ \sum\limits_{x\in S^{-1}} \omega_6^5 E_{a,x}=\sum\limits_{x\in S}\left( \omega_6 \psi_a(x)+ \omega_6^5 \psi_a(-x)\right).\]
Thus $(Ev)_a$ is an HS-eigenvalue of the integral oriented Cayley graph $\text{Cay}(\Gamma, S)$ for each $a\in \Gamma$. Therefore $Ev \in \mathbb{Q}^n$, and hence all the three conditions of Lemma ~\ref{Ori4Nec} hold. 
 
By the third condition of Lemma ~\ref{Ori4Nec}, $v_x=0$ for all $x\in \Gamma \setminus\Gamma(3)$, and so we must have $S \cup S^{-1} \subseteq \Gamma(3)$. Again, let  $x \in S$, $y \in \Gamma(3)$ and $ x \approx y$. The second condition of Lemma ~\ref{Ori4Nec} gives $v_x=v_y$, which implies that $ y\in S$. Thus $x \in S$ implies $\langle\!\langle x \rangle\!\rangle \subseteq S$. Hence $S\in \mathbb{E}(\Gamma)$. The converse part follows from Theorem~\ref{OrientedChara}.
\end{proof}

The following example illustrates Theorem~\ref{neccori}.

\begin{ex}\label{ex1} Consider $\Gamma= \mathbb{Z}_3 \times \mathbb{Z}_3$ and $S=\{ (0,1), (2,0)\}$. The oriented graph $\text{Cay}(\mathbb{Z}_3 \times \mathbb{Z}_3, S)$ is shown in Figure~\ref{a}. We see that $\langle\!\langle (0,1) \rangle\!\rangle=\{(0,1)\}$ and $\langle\!\langle (2,0)\rangle\!\rangle=\{(2,0)\}$. Therefore $S \in \mathbb{E}(\Gamma)$. Further, using Corollary~\ref{OriEig}  and Equation~\ref{character}, the HS-eigenvalues of  $\text{Cay}(\mathbb{Z}_3 \times \mathbb{Z}_3, S)$ are obtained as
$$\mu_\alpha= [\omega_6\psi_{\alpha}(0,1) + \omega_6^5 \psi_{\alpha}(0,2)] + [\omega_6 \psi_{\alpha}(2,0) + \omega_6^5 \psi_{\alpha}(1,0)] ~~\text{for each }\alpha \in \mathbb{Z}_3 \times \mathbb{Z}_3 ,$$
where 
$$\psi_{\alpha}(x) =\omega_3^{\alpha_1 x_1}\omega_3^{\alpha_2 x_2}  \text{ for all } \alpha=(\alpha_1,\alpha_2),x=(x_1,x_2)\in \mathbb{Z}_3 \times \mathbb{Z}_3.$$
It can be seen that $\mu_{(0,0)}=2,\mu_{(0,1)}=-1,\mu_{(0,2)}=2,\mu_{(1,0)}=2,\mu_{(1,1)}=-1,\mu_{(1,2)}=2$, $\mu_{(2,0)}=-1$, $\mu_{(2,1)}=-4$ and $\mu_{(2,2)}=-1$. Thus $\text{Cay}(\mathbb{Z}_3 \times \mathbb{Z}_3, S)$ is HS-integral. 
\end{ex}

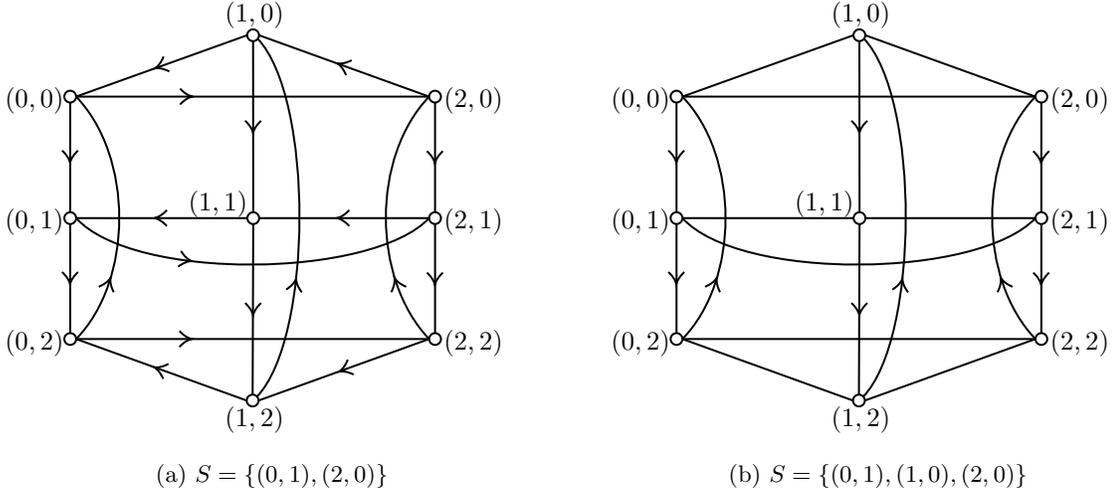
\begin{figure}[ht]
\centering
\hfill
\begin{subfigure}{0.45\textwidth}
\tikzset{every picture/.style={line width=0.75pt}} 

\begin{tikzpicture}[x=0.23pt,y=0.23pt,yscale=-1,xscale=1]

\draw  [color={rgb, 255:red, 0; green, 0; blue, 0 }  ,draw opacity=1 ][line width=0.75]  (190.54,200.79) .. controls (190.54,195.43) and (194.89,191.08) .. (200.25,191.08) .. controls (205.61,191.08) and (209.96,195.43) .. (209.96,200.79) .. controls (209.96,206.15) and (205.61,210.5) .. (200.25,210.5) .. controls (194.89,210.5) and (190.54,206.15) .. (190.54,200.79) -- cycle ;
\draw [line width=0.75]    (200.25,210.5) -- (200.29,391.04) ;
\draw  [line width=0.75]  (213.15,289.5) .. controls (205.86,296.28) and (201.49,303.06) .. (200.03,309.84) .. controls (198.57,303.06) and (194.19,296.28) .. (186.9,289.5) ;
\draw [line width=0.75]    (200.29,410.46) -- (200.29,590.04) ;
\draw [line width=0.75]    (800.25,210.5) -- (800.29,391.04) ;
\draw [line width=0.75]    (800.29,410.46) -- (800.29,590.04) ;
\draw [line width=0.75]    (501,109.46) -- (501.25,391.08) ;
\draw [line width=0.75]    (501.25,410.5) -- (500.25,691.29) ;
\draw    (209.96,200.79) .. controls (304.54,299.46) and (304.58,502.42) .. (210,599.75) ;
\draw    (790.58,599.75) .. controls (695.67,500.42) and (695.62,299.46) .. (790.54,200.79) ;
\draw    (508.58,105.42) .. controls (601.17,204.08) and (599.58,597.42) .. (507.58,694.42) ;
\draw    (209.96,200.79) -- (491.29,99.75) ;
\draw    (210,400.75) -- (491.54,400.79) ;
\draw    (210,599.75) -- (490.54,701) ;
\draw    (510.71,99.75) -- (790.54,200.79) ;
\draw    (510.96,400.79) -- (790.58,400.75) ;
\draw    (509.96,701) -- (790.58,599.75) ;
\draw    (209.96,200.79) -- (790.54,200.79) ;
\draw    (210,599.75) -- (790.58,599.75) ;
\draw    (210,400.75) .. controls (304.58,505) and (696.67,500.42) .. (790.58,400.75) ;
\draw  [color={rgb, 255:red, 0; green, 0; blue, 0 }  ,draw opacity=1 ][line width=0.75]  (190.58,400.75) .. controls (190.58,395.39) and (194.93,391.04) .. (200.29,391.04) .. controls (205.65,391.04) and (210,395.39) .. (210,400.75) .. controls (210,406.11) and (205.65,410.46) .. (200.29,410.46) .. controls (194.93,410.46) and (190.58,406.11) .. (190.58,400.75) -- cycle ;
\draw  [color={rgb, 255:red, 0; green, 0; blue, 0 }  ,draw opacity=1 ][line width=0.75]  (190.58,599.75) .. controls (190.58,594.39) and (194.93,590.04) .. (200.29,590.04) .. controls (205.65,590.04) and (210,594.39) .. (210,599.75) .. controls (210,605.11) and (205.65,609.46) .. (200.29,609.46) .. controls (194.93,609.46) and (190.58,605.11) .. (190.58,599.75) -- cycle ;
\draw  [color={rgb, 255:red, 0; green, 0; blue, 0 }  ,draw opacity=1 ][line width=0.75]  (790.54,200.79) .. controls (790.54,195.43) and (794.89,191.08) .. (800.25,191.08) .. controls (805.61,191.08) and (809.96,195.43) .. (809.96,200.79) .. controls (809.96,206.15) and (805.61,210.5) .. (800.25,210.5) .. controls (794.89,210.5) and (790.54,206.15) .. (790.54,200.79) -- cycle ;
\draw  [color={rgb, 255:red, 0; green, 0; blue, 0 }  ,draw opacity=1 ][line width=0.75]  (790.58,400.75) .. controls (790.58,395.39) and (794.93,391.04) .. (800.29,391.04) .. controls (805.65,391.04) and (810,395.39) .. (810,400.75) .. controls (810,406.11) and (805.65,410.46) .. (800.29,410.46) .. controls (794.93,410.46) and (790.58,406.11) .. (790.58,400.75) -- cycle ;
\draw  [color={rgb, 255:red, 0; green, 0; blue, 0 }  ,draw opacity=1 ][line width=0.75]  (790.58,599.75) .. controls (790.58,594.39) and (794.93,590.04) .. (800.29,590.04) .. controls (805.65,590.04) and (810,594.39) .. (810,599.75) .. controls (810,605.11) and (805.65,609.46) .. (800.29,609.46) .. controls (794.93,609.46) and (790.58,605.11) .. (790.58,599.75) -- cycle ;
\draw  [color={rgb, 255:red, 0; green, 0; blue, 0 }  ,draw opacity=1 ][line width=0.75]  (491.29,99.75) .. controls (491.29,94.39) and (495.64,90.04) .. (501,90.04) .. controls (506.36,90.04) and (510.71,94.39) .. (510.71,99.75) .. controls (510.71,105.11) and (506.36,109.46) .. (501,109.46) .. controls (495.64,109.46) and (491.29,105.11) .. (491.29,99.75) -- cycle ;
\draw  [color={rgb, 255:red, 0; green, 0; blue, 0 }  ,draw opacity=1 ][line width=0.75]  (491.54,400.79) .. controls (491.54,395.43) and (495.89,391.08) .. (501.25,391.08) .. controls (506.61,391.08) and (510.96,395.43) .. (510.96,400.79) .. controls (510.96,406.15) and (506.61,410.5) .. (501.25,410.5) .. controls (495.89,410.5) and (491.54,406.15) .. (491.54,400.79) -- cycle ;
\draw  [color={rgb, 255:red, 0; green, 0; blue, 0 }  ,draw opacity=1 ][line width=0.75]  (490.54,701) .. controls (490.54,695.64) and (494.89,691.29) .. (500.25,691.29) .. controls (505.61,691.29) and (509.96,695.64) .. (509.96,701) .. controls (509.96,706.36) and (505.61,710.71) .. (500.25,710.71) .. controls (494.89,710.71) and (490.54,706.36) .. (490.54,701) -- cycle ;
\draw  [line width=0.75]  (213.15,489.5) .. controls (205.86,496.28) and (201.49,503.06) .. (200.03,509.84) .. controls (198.57,503.06) and (194.19,496.28) .. (186.9,489.5) ;
\draw  [line width=0.75]  (247.21,513.74) .. controls (256.39,509.86) and (262.81,504.98) .. (266.51,499.12) .. controls (265.56,505.98) and (267.36,513.85) .. (271.88,522.72) ;
\draw  [line width=0.75]  (728.17,522.72) .. controls (732.71,513.85) and (734.49,505.98) .. (733.55,499.12) .. controls (737.24,504.99) and (743.67,509.86) .. (752.84,513.74) ;
\draw  [line width=0.75]  (514.15,240.5) .. controls (506.86,247.28) and (502.49,254.06) .. (501.03,260.84) .. controls (499.57,254.06) and (495.19,247.28) .. (487.9,240.5) ;
\draw  [line width=0.75]  (514.15,540.5) .. controls (506.86,547.28) and (502.49,554.06) .. (501.03,560.84) .. controls (499.57,554.06) and (495.19,547.28) .. (487.9,540.5) ;
\draw  [line width=0.75]  (555.33,518.41) .. controls (563.7,513) and (569.18,507.08) .. (571.79,500.66) .. controls (572.05,507.58) and (575.19,515.02) .. (581.19,522.97) ;
\draw  [line width=0.75]  (363.2,159.32) .. controls (354.27,154.92) and (346.38,153.24) .. (339.52,154.29) .. controls (345.34,150.52) and (350.12,144.01) .. (353.86,134.79) ;
\draw  [line width=0.75]  (653.96,165.53) .. controls (650.16,156.32) and (645.34,149.85) .. (639.5,146.11) .. controls (646.36,147.12) and (654.24,145.39) .. (663.15,140.93) ;
\draw  [line width=0.75]  (380.86,187.55) .. controls (387.64,194.84) and (394.42,199.21) .. (401.2,200.67) .. controls (394.42,202.13) and (387.64,206.51) .. (380.86,213.8) ;
\draw  [line width=0.75]  (382.42,454.91) .. controls (387.76,463.32) and (393.63,468.85) .. (400.03,471.53) .. controls (393.1,471.72) and (385.63,474.79) .. (377.63,480.72) ;
\draw  [line width=0.75]  (379.86,586.55) .. controls (386.64,593.84) and (393.42,598.21) .. (400.2,599.67) .. controls (393.42,601.13) and (386.64,605.51) .. (379.86,612.8) ;
\draw  [line width=0.75]  (353.3,665.42) .. controls (349.3,656.29) and (344.34,649.93) .. (338.43,646.32) .. controls (345.31,647.18) and (353.15,645.28) .. (361.95,640.63) ;
\draw  [line width=0.75]  (666.29,658.17) .. controls (657.31,653.86) and (649.4,652.27) .. (642.56,653.39) .. controls (648.34,649.55) and (653.05,643) .. (656.7,633.74) ;
\draw  [line width=0.75]  (359.2,413.8) .. controls (352.42,406.5) and (345.64,402.13) .. (338.86,400.67) .. controls (345.64,399.21) and (352.42,394.83) .. (359.2,387.55) ;
\draw  [line width=0.75]  (659.2,413.8) .. controls (652.42,406.5) and (645.64,402.13) .. (638.86,400.67) .. controls (645.64,399.21) and (652.42,394.83) .. (659.2,387.55) ;
\draw  [line width=0.75]  (813.15,290.5) .. controls (805.86,297.28) and (801.49,304.06) .. (800.03,310.84) .. controls (798.57,304.06) and (794.19,297.28) .. (786.9,290.5) ;
\draw  [line width=0.75]  (813.15,490.5) .. controls (805.86,497.28) and (801.49,504.06) .. (800.03,510.84) .. controls (798.57,504.06) and (794.19,497.28) .. (786.9,490.5) ;

\draw (88,180) node [anchor=north west][inner sep=0.75pt]    {$( 0,0)$};
\draw (88,380) node [anchor=north west][inner sep=0.75pt]    {$( 0,1)$};
\draw (88,580) node [anchor=north west][inner sep=0.75pt]    {$( 0,2)$};
\draw (810,180) node [anchor=north west][inner sep=0.75pt]    {$( 2,0)$};
\draw (810,380) node [anchor=north west][inner sep=0.75pt]    {$( 2,1)$};
\draw (810,580) node [anchor=north west][inner sep=0.75pt]    {$( 2,2)$};
\draw (450,40) node [anchor=north west][inner sep=0.75pt]    {$( 1,0)$};
\draw (390,350) node [anchor=north west][inner sep=0.75pt]    {$( 1,1)$};
\draw (450,706.4) node [anchor=north west][inner sep=0.75pt]    {$( 1,2)$};
\end{tikzpicture}
\caption{$S=\{ (0,1), (2,0)\}$}\label{a}
\end{subfigure}
\hfill
\begin{subfigure}{0.45\textwidth}

\tikzset{every picture/.style={line width=0.75pt}} 
\begin{tikzpicture}[x=0.23pt,y=0.23pt,yscale=-1,xscale=1]

\draw  [color={rgb, 255:red, 0; green, 0; blue, 0 }  ,draw opacity=1 ][line width=0.75]  (190.54,200.79) .. controls (190.54,195.43) and (194.89,191.08) .. (200.25,191.08) .. controls (205.61,191.08) and (209.96,195.43) .. (209.96,200.79) .. controls (209.96,206.15) and (205.61,210.5) .. (200.25,210.5) .. controls (194.89,210.5) and (190.54,206.15) .. (190.54,200.79) -- cycle ;
\draw [line width=0.75]    (200.25,210.5) -- (200.29,391.04) ;
\draw  [line width=0.75]  (213.15,289.5) .. controls (205.86,296.28) and (201.49,303.06) .. (200.03,309.84) .. controls (198.57,303.06) and (194.19,296.28) .. (186.9,289.5) ;
\draw [line width=0.75]    (200.29,410.46) -- (200.29,590.04) ;
\draw [line width=0.75]    (800.25,210.5) -- (800.29,391.04) ;
\draw [line width=0.75]    (800.29,410.46) -- (800.29,590.04) ;
\draw [line width=0.75]    (501,109.46) -- (501.25,391.08) ;
\draw [line width=0.75]    (501.25,410.5) -- (500.25,691.29) ;
\draw    (209.96,200.79) .. controls (304.54,299.46) and (304.58,502.42) .. (210,599.75) ;
\draw    (790.58,599.75) .. controls (695.67,500.42) and (695.62,299.46) .. (790.54,200.79) ;
\draw    (508.58,105.42) .. controls (601.17,204.08) and (599.58,597.42) .. (507.58,694.42) ;
\draw    (209.96,200.79) -- (491.29,99.75) ;
\draw    (210,400.75) -- (491.54,400.79) ;
\draw    (210,599.75) -- (490.54,701) ;
\draw    (510.71,99.75) -- (790.54,200.79) ;
\draw    (510.96,400.79) -- (790.58,400.75) ;
\draw    (509.96,701) -- (790.58,599.75) ;
\draw    (209.96,200.79) -- (790.54,200.79) ;
\draw    (210,599.75) -- (790.58,599.75) ;
\draw    (210,400.75) .. controls (304.58,505) and (696.67,500.42) .. (790.58,400.75) ;
\draw  [color={rgb, 255:red, 0; green, 0; blue, 0 }  ,draw opacity=1 ][line width=0.75]  (190.58,400.75) .. controls (190.58,395.39) and (194.93,391.04) .. (200.29,391.04) .. controls (205.65,391.04) and (210,395.39) .. (210,400.75) .. controls (210,406.11) and (205.65,410.46) .. (200.29,410.46) .. controls (194.93,410.46) and (190.58,406.11) .. (190.58,400.75) -- cycle ;
\draw  [color={rgb, 255:red, 0; green, 0; blue, 0 }  ,draw opacity=1 ][line width=0.75]  (190.58,599.75) .. controls (190.58,594.39) and (194.93,590.04) .. (200.29,590.04) .. controls (205.65,590.04) and (210,594.39) .. (210,599.75) .. controls (210,605.11) and (205.65,609.46) .. (200.29,609.46) .. controls (194.93,609.46) and (190.58,605.11) .. (190.58,599.75) -- cycle ;
\draw  [color={rgb, 255:red, 0; green, 0; blue, 0 }  ,draw opacity=1 ][line width=0.75]  (790.54,200.79) .. controls (790.54,195.43) and (794.89,191.08) .. (800.25,191.08) .. controls (805.61,191.08) and (809.96,195.43) .. (809.96,200.79) .. controls (809.96,206.15) and (805.61,210.5) .. (800.25,210.5) .. controls (794.89,210.5) and (790.54,206.15) .. (790.54,200.79) -- cycle ;
\draw  [color={rgb, 255:red, 0; green, 0; blue, 0 }  ,draw opacity=1 ][line width=0.75]  (790.58,400.75) .. controls (790.58,395.39) and (794.93,391.04) .. (800.29,391.04) .. controls (805.65,391.04) and (810,395.39) .. (810,400.75) .. controls (810,406.11) and (805.65,410.46) .. (800.29,410.46) .. controls (794.93,410.46) and (790.58,406.11) .. (790.58,400.75) -- cycle ;
\draw  [color={rgb, 255:red, 0; green, 0; blue, 0 }  ,draw opacity=1 ][line width=0.75]  (790.58,599.75) .. controls (790.58,594.39) and (794.93,590.04) .. (800.29,590.04) .. controls (805.65,590.04) and (810,594.39) .. (810,599.75) .. controls (810,605.11) and (805.65,609.46) .. (800.29,609.46) .. controls (794.93,609.46) and (790.58,605.11) .. (790.58,599.75) -- cycle ;
\draw  [color={rgb, 255:red, 0; green, 0; blue, 0 }  ,draw opacity=1 ][line width=0.75]  (491.29,99.75) .. controls (491.29,94.39) and (495.64,90.04) .. (501,90.04) .. controls (506.36,90.04) and (510.71,94.39) .. (510.71,99.75) .. controls (510.71,105.11) and (506.36,109.46) .. (501,109.46) .. controls (495.64,109.46) and (491.29,105.11) .. (491.29,99.75) -- cycle ;
\draw  [color={rgb, 255:red, 0; green, 0; blue, 0 }  ,draw opacity=1 ][line width=0.75]  (491.54,400.79) .. controls (491.54,395.43) and (495.89,391.08) .. (501.25,391.08) .. controls (506.61,391.08) and (510.96,395.43) .. (510.96,400.79) .. controls (510.96,406.15) and (506.61,410.5) .. (501.25,410.5) .. controls (495.89,410.5) and (491.54,406.15) .. (491.54,400.79) -- cycle ;
\draw  [color={rgb, 255:red, 0; green, 0; blue, 0 }  ,draw opacity=1 ][line width=0.75]  (490.54,701) .. controls (490.54,695.64) and (494.89,691.29) .. (500.25,691.29) .. controls (505.61,691.29) and (509.96,695.64) .. (509.96,701) .. controls (509.96,706.36) and (505.61,710.71) .. (500.25,710.71) .. controls (494.89,710.71) and (490.54,706.36) .. (490.54,701) -- cycle ;
\draw  [line width=0.75]  (213.15,489.5) .. controls (205.86,496.28) and (201.49,503.06) .. (200.03,509.84) .. controls (198.57,503.06) and (194.19,496.28) .. (186.9,489.5) ;
\draw  [line width=0.75]  (247.21,513.74) .. controls (256.39,509.86) and (262.81,504.98) .. (266.51,499.12) .. controls (265.56,505.98) and (267.36,513.85) .. (271.88,522.72) ;
\draw  [line width=0.75]  (728.17,522.72) .. controls (732.71,513.85) and (734.49,505.98) .. (733.55,499.12) .. controls (737.24,504.99) and (743.67,509.86) .. (752.84,513.74) ;
\draw  [line width=0.75]  (514.15,240.5) .. controls (506.86,247.28) and (502.49,254.06) .. (501.03,260.84) .. controls (499.57,254.06) and (495.19,247.28) .. (487.9,240.5) ;
\draw  [line width=0.75]  (514.15,540.5) .. controls (506.86,547.28) and (502.49,554.06) .. (501.03,560.84) .. controls (499.57,554.06) and (495.19,547.28) .. (487.9,540.5) ;
\draw  [line width=0.75]  (555.33,518.41) .. controls (563.7,513) and (569.18,507.08) .. (571.79,500.66) .. controls (572.05,507.58) and (575.19,515.02) .. (581.19,522.97) ;
\draw  [line width=0.75]  (813.15,290.5) .. controls (805.86,297.28) and (801.49,304.06) .. (800.03,310.84) .. controls (798.57,304.06) and (794.19,297.28) .. (786.9,290.5) ;
\draw  [line width=0.75]  (813.15,490.5) .. controls (805.86,497.28) and (801.49,504.06) .. (800.03,510.84) .. controls (798.57,504.06) and (794.19,497.28) .. (786.9,490.5) ;

\draw (88,180) node [anchor=north west][inner sep=0.75pt]    {$( 0,0)$};
\draw (88,380) node [anchor=north west][inner sep=0.75pt]    {$( 0,1)$};
\draw (88,580) node [anchor=north west][inner sep=0.75pt]    {$( 0,2)$};
\draw (810,180) node [anchor=north west][inner sep=0.75pt]    {$( 2,0)$};
\draw (810,380) node [anchor=north west][inner sep=0.75pt]    {$( 2,1)$};
\draw (810,580) node [anchor=north west][inner sep=0.75pt]    {$( 2,2)$};
\draw (450,40) node [anchor=north west][inner sep=0.75pt]    {$( 1,0)$};
\draw (390,350) node [anchor=north west][inner sep=0.75pt]    {$( 1,1)$};
\draw (450,706.4) node [anchor=north west][inner sep=0.75pt]    {$( 1,2)$};
\end{tikzpicture}
\caption{$S=\{ (0,1),(1,0), (2,0)\}$} \label{b}
\end{subfigure}
\caption{The graph $\text{Cay}(\mathbb{Z}_3 \times \mathbb{Z}_3, S)$}\label{main}
\end{figure}

\begin{lema}\label{CharaNewIntegSum} Let $S$ be a skew-symmetric subset of an abelian group $\Gamma$ and $t(\neq 0) \in \mathbb{Q}$. If \linebreak[4] $\sum\limits_{s\in S}it\sqrt{3} (\psi_{\alpha}(s)- \psi_{\alpha}(-s))$ is an integer $ \textnormal{ for each } \alpha \in \Gamma$ then $S \in \mathbb{E}(\Gamma)$
	\end{lema}
	\begin{proof}
Let $v$ be the vector, whose coordinates are indexed by the elements of $\Gamma$, defined by
			$$v_x= \left\{ \begin{array}{cl}
				it\sqrt{3} & \mbox{if }  x\in S \\
				-it\sqrt{3} & \mbox{if }  x\in S^{-1}\\ 
				0 &  \mbox{otherwise}. 
			\end{array}\right.$$ 
Since $v\in \mathbb{Q}(\omega_3)^n$ and  $\alpha$-th coordinate of $Ev$ is $\sum\limits_{s\in S} it\sqrt{3} ( \psi_{\alpha}(s)-\psi_{\alpha}(-s))$, we have $Ev \in \mathbb{Q}^n$. By the third condition of Lemma ~\ref{Ori4Nec}, $\Im(v_x)=0$, and so $v_x=0$ for all $x\in \Gamma \setminus\Gamma(3)$. Thus we must have $S \cup S^{-1} \subseteq \Gamma(3)$. Again, let  $x \in S$, $y \in \Gamma(3)$ and $ x \approx y$. The second condition of Lemma ~\ref{Ori4Nec} gives $v_x=v_y$, which implies that $ y\in S$. Thus $x \in S$ implies $\langle\!\langle x \rangle\!\rangle \subseteq S$. Hence $S\in \mathbb{E}(\Gamma)$. 
	\end{proof}
	
		\begin{lema}\label{Sqrt3NecessIntSum} Let $S$ be a skew-symmetric subset of an abelian group $\Gamma$ and $t(\neq 0) \in \mathbb{Q}$. If \linebreak[4] $\sum\limits_{s\in S}it\sqrt{3} ( \psi_{\alpha}(s)- \psi_{\alpha}(-s))$ is an integer for each $ \alpha \in \Gamma$ then $\sum\limits_{s\in S\cup S^{-1}} \psi_{\alpha}(s) $ is an integer for each $ \alpha \in \Gamma$.
	\end{lema}
	\begin{proof}
Assume that $\sum\limits_{s\in S}it\sqrt{3} (\psi_{\alpha}(s)-\psi_{\alpha}(-s))$ is an integer for each $ \alpha \in \Gamma$. By Lemma \ref{CharaNewIntegSum} we have $S \in \mathbb{E}(\Gamma)$, and so $S=\langle\!\langle x_1 \rangle\!\rangle\cup...\cup \langle\!\langle x_k \rangle\!\rangle$ for some $x_1,...,x_k\in \Gamma(3)$. Therefore, using Lemma~\ref{lemanecc} we get $S \cup S^{-1}=[x_1] \cup ...\cup[x_k] \in \mathbb{B}(\Gamma)$. Thus by Theorem~\ref{Cayint}, $\text{Cay}(\Gamma, S \cup S^{-1})$ is integral, that is, $\sum\limits_{s\in S\cup S^{-1}} \psi_{\alpha}(s) $ is an integer for each $ \alpha \in \Gamma$. 
	\end{proof}

	\begin{lema}\label{SeperatIntegMixedGraph}
Let $\Gamma$ be an abelian group. The mixed Cayley graph $\text{Cay}(\Gamma,S)$ is HS-integral if and only if  $\text{Cay}(\Gamma,{S\setminus \overline{S}})$ is integral and $\text{Cay}(\Gamma, {\overline{S}})$ are HS-integral.
	\end{lema}
	\begin{proof}
	Assume that the mixed Cayley graph $\text{Cay}(\Gamma,S)$ is HS-integral. Let the HS-spectrum of $\text{Cay}(\Gamma,S)$ be $\{\gamma_{\alpha}: \alpha \in \Gamma\}$, where $\gamma_{\alpha}=\lambda_{\alpha} +\mu_{\alpha}$, 
	$$\lambda_{\alpha}= \sum\limits_{s \in S\setminus \overline{S}} \psi_{\alpha}(s) \textnormal{ and } \mu_{\alpha}=  \sum\limits_{s \in \overline{S}} (\omega_6 \psi_{\alpha}(s) + \omega_6^5 \psi_{\alpha}(-s)), \textnormal{ for } \alpha \in \Gamma.$$
Note that $\{\lambda_{\alpha}: \alpha \in \Gamma\}$ is the spectrum of $\text{Cay}(\Gamma, S\setminus \overline{S})$ and $\{\mu_{\alpha}: \alpha \in \Gamma\}$ is the HS-spectrum of $\text{Cay}(\Gamma,\overline{S})$.  By assumption $\gamma_{\alpha} \in \mathbb{Z}$, and so $ \gamma_{\alpha} - \gamma_{-\alpha}= \sum\limits_{s\in \overline{S}}i\sqrt{3}(\psi_{\alpha}(s)-\psi_{\alpha}(-s))  \in \mathbb{Z}$ for all $ \alpha \in \Gamma$. By Lemma \ref{Sqrt3NecessIntSum}, we get $\sum\limits_{ s \in \overline{S} \cup \overline{S}^{-1}}\psi_{\alpha}(s) \in \mathbb{Z}$ for all $ \alpha \in \Gamma$. Note that $\mu_{\alpha}$ is an algebraic integer. Also 
\begin{align*}
\mu_{\alpha}= \frac{1}{2} \sum\limits_{ s \in \overline{S} \cup \overline{S}^{-1}}\psi_{\alpha}(s) + \frac{1}{2} \sum\limits_{s\in \overline{S}}i\sqrt{3}(\psi_{\alpha}(s)-\psi_{\alpha}(-s))\in \mathbb{Q}.
\end{align*}
Hence $\mu_{\alpha}$ is an integer for each $ \alpha \in \Gamma$. Thus $\text{Cay}(\Gamma,\overline{S})$ is HS-integral. Now we have $\gamma_{\alpha}, \mu_{\alpha} \in \mathbb{Z}$, and so $\lambda_{\alpha} = \gamma_{\alpha} -\mu_{\alpha} \in \mathbb{Z}$ for each $ \alpha \in \Gamma$. Hence $\text{Cay}(\Gamma,S\setminus \overline{S})$ is integral.
			
			Conversely, assume that $\text{Cay}(\Gamma,S\setminus \overline{S})$ is integral and $\text{Cay}(\Gamma, \overline{S})$ is HS-integral. Then Lemma \ref{imcgoa3} implies that $\text{Cay}(\Gamma,S)$ is integral.
	\end{proof}
	
\begin{theorem}\label{CharaHSintMixed}
Let $\Gamma$ be an abelian group. The mixed Cayley graph  $\text{Cay}(\Gamma, S)$ is HS-integral if and only if $S \setminus \overline{S} \in \mathbb{B}(\Gamma)$ and $\overline{S} \in \mathbb{E}(\Gamma)$.
\end{theorem}
\begin{proof}
By Lemma~\ref{SeperatIntegMixedGraph}, the mixed Cayley graph $\text{Cay}(\Gamma, S)$ is HS-integral if and only if $\text{Cay}(\Gamma, S \setminus \overline{S})$ is integral and $\text{Cay}(\Gamma, \overline{S})$ is HS-integral. Note that $S \setminus \overline{S}$ is a symmetric set and $\overline{S}$ is a skew-symmetric set. Thus by Theorem~\ref{Cayint}, $\text{Cay}(\Gamma, S \setminus \overline{S})$ is integral if and only if $S \setminus \overline{S} \in \mathbb{B}(\Gamma)$. By Theorem~\ref{neccori}, $\text{Cay}(\Gamma, \overline{S})$ is HS-integral if and only if $\overline{S} \in \mathbb{E}(\Gamma)$. Hence the result follows.
\end{proof}

The following example illustrates Theorem~\ref{CharaHSintMixed}.

\begin{ex}\label{ex2} Consider $\Gamma= \mathbb{Z}_3 \times \mathbb{Z}_3$ and $S=\{ (0,1),(1,0), (2,0)\}$. The mixed graph  $\text{Cay}(\mathbb{Z}_3 \times \mathbb{Z}_3, S)$ is shown in Figure~\ref{b}. Here $\overline{S}=\{(0,1)\}=\langle\!\langle (0,1)\rangle\!\rangle\in \mathbb{E}(\Gamma)$ and $S\setminus\overline{S}=\{(1,0),(2,0)\} =[(1,0)]\in \mathbb{B}(\Gamma)$. Further, using Lemma~\ref{imcgoa3}  and Equation~\ref{character}, the HS-eigenvalues of  $\text{Cay}(\mathbb{Z}_3 \times \mathbb{Z}_3, S)$ are obtained as
$$\gamma_\alpha=[\psi_{\alpha}(1,0) + \psi_{\alpha}(2,0)] + [\omega_6 \psi_{\alpha}(0,1) + \omega_6^5 \psi_{\alpha}(0,2)] ~~\text{for each }\alpha \in \mathbb{Z}_3 \times \mathbb{Z}_3.$$
One can see that $\gamma_{(0,1)}=\gamma_{(1,0)}= \gamma_{(1,2)}= \gamma_{(2,0)}=\gamma_{(2,2)}=0$, $\gamma_{(0,0)}=\gamma_{(0,2)}=3$ and $\gamma_{(1,1)} = \gamma_{(2,1)}=-3$. Thus $\text{Cay}(\mathbb{Z}_3 \times \mathbb{Z}_3, S)$ is HS-integral.
\end{ex}

\section{Characterization of Eisenstein integral mixed Cayley graphs over abelian groups}\label{sec5}
Let $\Gamma$ be a finite abelian group of order $n$. For an $S\subseteq \Gamma$ with $0\notin S$, consider the function $\alpha:\Gamma \rightarrow \{0,1\}$ defined by 
\[\alpha(s)= \left\{ \begin{array}{rl}
		1 & \mbox{if } s\in S \\
		0 &   \mbox{otherwise}, 
	\end{array}\right.\]
in Theorem~\ref{EigNorColCayMix}. We see that $\sum\limits_{s \in S} \psi_{\alpha}(s)$ is an eigenvalue of the mixed Cayley graph $\text{Cay}(\Gamma, S)$ for all $\alpha \in \Gamma$. For $x \in \Gamma, y\in \Gamma(3)$ and $\alpha \in \Gamma$, define 
$$C_x(\alpha)=\sum_{s \in [x] } \psi_{\alpha}(s)~~\text{ and }~~ T_y(\alpha)= \sum_{s \in \langle\!\langle y \rangle\!\rangle} i\sqrt{3}( \psi_{\alpha}(s)-\psi_{\alpha}(-s)).$$
  Note that $C_x(\alpha)$ is an eigenvalue of the mixed Cayley graph $\text{Cay}(\Gamma, [x])$ for each $\alpha \in \Gamma$.

\begin{lema}\label{Tn(q)IsIntegerForAll}
Let $x \in \Gamma(3)$. Then $T_x(\alpha)$ is an integer for each $\alpha \in \Gamma$.
\end{lema}
\begin{proof}
We have
\begin{equation*}
			\begin{split}
Z_x(\alpha) 	= \sum\limits_{s \in \langle\!\langle x \rangle\!\rangle}\left( \omega_6\psi_{\alpha}(s) + \omega_6^5 \psi_{\alpha}(-s)\right)
			&= \frac{1}{2} \sum\limits_{s\in [x]}\psi_{\alpha}(s) + \frac{i\sqrt{3}}{2} \sum\limits_{s \in \langle\!\langle x \rangle\!\rangle} (\psi_{\alpha}(s) - \psi_{\alpha}(-s))\\
			&= \frac{C_x(\alpha)}{2} + \frac{T_x(\alpha)}{2}.\\
			\end{split} 
	\end{equation*}
By Lemma~\ref{integerEigenvalue}, $T_x(\alpha)=  2 Z_x(\alpha) - C_x(\alpha)$ is an integer for each $\alpha \in \Gamma$.
\end{proof}

\begin{lema} Let $\Gamma$ be a finite abelian group and the order of $x\in \Gamma(3)$ be $3m$, with $m \not\equiv 0 \Mod 3$. Then
$$x^m [x^3] = \left\{ \begin{array}{ll}
			\langle\!\langle x \rangle\!\rangle & \mbox{if } m \equiv 1 \Mod 3  \\
			\langle\!\langle -x \rangle\!\rangle & \mbox{if } m \equiv 2 \Mod 3.
		\end{array}\right. $$
\end{lema}
\begin{proof}
Assume that $m \equiv 1 \Mod 3$. Let $y \in x^m [x^3]$. Then $y=x^{m+3r}$ for some $r\in G_m(1)$. We have $\gcd(r,m)=1$, which implies that $\gcd(m+3r, 3m)=1$ and $m+3r \equiv 1 \Mod 3$. Thus $x^m [x^3] \subseteq \langle\!\langle x \rangle\!\rangle$. Since size of both $x^m [x^3]$ and $\langle\!\langle x \rangle\!\rangle$ are same, so $x^m [x^3] = \langle\!\langle x \rangle\!\rangle$. Similarly, if $m \equiv 2 \Mod 3$ then we have $x^m [x^3] = \langle\!\langle -x \rangle\!\rangle$.
\end{proof}
	
\begin{lema}\label{Sum3mequalRamanujan}
Let $\Gamma$ be a finite abelian group and the order of $x\in \Gamma(3)$ be $3m$, with $m \not\equiv 0 \Mod 3$. Then 
\begin{equation*}
	\begin{split}
		T_x(\alpha) =  \left\{ \begin{array}{cl}
			\pm3 C_{x^3}(\alpha) & \mbox{if } \psi_{\alpha}(x^m) \neq 1 \\
			0 & \mbox{otherwise.}
		\end{array}\right.		
	\end{split} 
\end{equation*} for all $\alpha \in \Gamma$.
\end{lema}
	\begin{proof}
		We have 
		\begin{equation*}
			\begin{split}
				T_x(\alpha) &= i\sqrt{3}   \sum\limits_{s \in \langle\!\langle x \rangle\!\rangle} (\psi_{\alpha}(s) - \psi_{\alpha}(-s)) \\
				&= \left\{ \begin{array}{rl}
					i\sqrt{3}  \sum\limits_{s \in \langle\!\langle x \rangle\!\rangle} (\psi_{\alpha}(s) - \psi_{\alpha}(-s)) & \mbox{if } m \equiv 1 \Mod 3 \\
					-i\sqrt{3}  \sum\limits_{s \in \langle\!\langle -x \rangle\!\rangle} (\psi_{\alpha}(s) - \psi_{\alpha}(-s)) & \mbox{if } m \equiv 2 \Mod 3
				\end{array}\right.\\
				&= \left\{ \begin{array}{rl}
					i\sqrt{3}  \sum\limits_{s \in x^m [x^3]} (\psi_{\alpha}(s) - \psi_{\alpha}(-s)) & \mbox{if } m \equiv 1 \Mod 3 \\
					-i\sqrt{3}  \sum\limits_{s \in x^m [x^3]} (\psi_{\alpha}(s) - \psi_{\alpha}(-s)) & \mbox{if } m \equiv 2 \Mod 3
				\end{array}\right.\\
				&= \left\{ \begin{array}{rl}
					i\sqrt{3}  \sum\limits_{s \in  [x^3]} (\psi_{\alpha}(x^m) \psi_{\alpha}(s) - \psi_{\alpha}(-x^{m}) \psi_{\alpha}(-s)) & \mbox{if } m \equiv 1 \Mod 3 \\
					-i\sqrt{3}  \sum\limits_{s \in [x^3]} (\psi_{\alpha}(x^m) \psi_{\alpha}(s) - \psi_{\alpha}(-x^m) \psi_{\alpha}(-s)) & \mbox{if } m \equiv 2 \Mod 3
				\end{array}\right.\\
				&= \left\{ \begin{array}{rl}
					-2\sqrt{3}  \Im(\psi_{\alpha}(x^m)) \sum\limits_{s \in  [x^3]} \psi_{\alpha}(s) & \mbox{if } m \equiv 1 \Mod 3 \\
					2\sqrt{3}  \Im(\psi_{\alpha}(x^m)) \sum\limits_{s \in [x^3]} \psi_{\alpha}(s) & \mbox{if } m \equiv 2 \Mod 3
				\end{array}\right.\\
				&=\pm 2\sqrt{3}  \Im(\psi_{\alpha}(x^m)) C_{x^3}(\alpha).
			\end{split} 
		\end{equation*} 
Since $\psi_{\alpha}(x^m)$ is a $3$-rd root of unity, $\Im(\psi_{\alpha}(x^m))=0$ or $\pm \frac{ \sqrt{3}}{2}$. Thus 
\begin{equation*}
	\begin{split}
T_x(\alpha) =  \left\{ \begin{array}{cl}
	\pm3 C_{x^3}(\alpha) & \mbox{if } \psi_{\alpha}(x^m) \neq 1 \\
	0 & \mbox{otherwise.}
\end{array}\right.		
	\end{split} 
\end{equation*}	
	\end{proof}

	\begin{lema}\label{Tn(q)SumEqualTo3TimesSum}
	Let $\Gamma$ be a finite abelian group and the order of $x\in \Gamma(3)$ be $k=3^tm$, with $m \not\equiv 0 \Mod 3$ and $t\geq 2$. Then $$T_x(\alpha)= \left\{ \begin{array}{ll}
					3 \sqrt{3} i \sum\limits_{r \in G_{3^{t-1}m,3}^1(1)} (\psi_{\alpha}(x^r) - \psi_{\alpha}(-x^r))  & \mbox{if } \psi_{\alpha}(x^{\frac{k}{3}}) =1 \\
					0 	  & otherwise.
				\end{array}\right.$$
	\end{lema}
	\begin{proof} Note that $ G_{k,3}^1(1)= G_{\frac{k}{3},3}^1(1) \cup \left(\frac{k}{3} + G_{\frac{k}{3},3}^1(1)\right) \cup \left(\frac{2k}{3} + G_{\frac{k}{3},3}^1(1)\right)$. Therefore
	\begin{equation*}
		\begin{split}
		T_x(\alpha) =& i\sqrt{3}  \sum\limits_{s \in \langle\!\langle x \rangle\!\rangle} (\psi_{\alpha}(s) - \psi_{\alpha}(-s)) \\
		=& i\sqrt{3}  \sum\limits_{r \in G_{k,3}^1(1) } (\psi_{\alpha}(x^r) - \psi_{\alpha}(-x^r)) \\ 	
		=& 	i\sqrt{3} \bigg[\sum\limits_{r \in G_{\frac{k}{3},3}^1(1)} (\psi_{\alpha}(x^r) - \psi_{\alpha}(-x^r)) +  \sum\limits_{r \in G_{\frac{k}{3},3}^1(1)} (\psi_{\alpha}(x^{\frac{k}{3}}) \psi_{\alpha}(x^r) - \psi_{\alpha}(x^{\frac{2k}{3}}) \psi_{\alpha}(-x^r)) \\
		&+ \sum\limits_{r \in G_{\frac{k}{3},3}^1(1)} (\psi_{\alpha}(x^{\frac{2k}{3}}) \psi_{\alpha}(x^r) - \psi_{\alpha}(x^{\frac{k}{3}}) \psi_{\alpha}(-x^r))\bigg]\\
		=& 	i\sqrt{3} \bigg[ \sum\limits_{r \in G_{\frac{k}{3},3}^1(1) }  (\psi_{\alpha}(x^r) - \psi_{\alpha}(-x^r)) +  \psi_{\alpha}(x^{\frac{k}{3}}) \sum\limits_{r \in G_{\frac{k}{3},3}^1(1) } ( \psi_{\alpha}(x^r) - \psi_{\alpha}(-x^r)) \\
		&+\psi_{\alpha}(x^{\frac{2k}{3}}) \sum\limits_{r \in G_{\frac{k}{3},3}^1(1)} ( \psi_{\alpha}(x^r) - \psi_{\alpha}(-x^r))\bigg]\\
		=& i\sqrt{3} (1+ \psi_{\alpha}(x^{\frac{k}{3}}) + \psi_{\alpha}(x^{\frac{2k}{3}})) \sum\limits_{r \in G_{\frac{k}{3},3}^1(1)}  (\psi_{\alpha}(x^r) - \psi_{\alpha}(-x^r))\\
		=& \left\{ \begin{array}{cl}
					3 \sqrt{3} i \sum\limits_{r \in G_{\frac{k}{3},3}^1(1)} (\psi_{\alpha}(x^r) - \psi_{\alpha}(-x^r))  & \mbox{if } \psi_{\alpha}(x^{\frac{k}{3}}) =1 \\
					0 	  & otherwise.
				\end{array}\right.
		\end{split} 
	\end{equation*}
\end{proof}

		\begin{lema}\label{3DividesTn(q)}
	Let $\Gamma$ be a finite abelian group and $x\in \Gamma(3)$. Then $\frac{T_x(\alpha)}{3}$ is an integer for each $\alpha \in \Gamma $.
	\end{lema}
	\begin{proof}
	Let $x\in \Gamma(3)$ and order of $x$ be $k=3^tm$ with $m\not\equiv 0 \Mod 3$ and $t\geq 1$. If $t=1$ then by Lemma~\ref{Sum3mequalRamanujan}, $\frac{T_x(\alpha)}{3}$ is an integer for each $\alpha \in \Gamma $. Assume that $t\geq 2$. If $\psi_{\alpha}(x^{\frac{k}{3}}) \neq1$ then by Lemma~\ref{Tn(q)SumEqualTo3TimesSum}, $\frac{T_x(\alpha)}{3}$ is an integer for each $\alpha \in \Gamma $. If $\psi_{\alpha}(x^{\frac{k}{3}})=1$ then by Lemma~\ref{Tn(q)IsIntegerForAll} and Lemma~\ref{Tn(q)SumEqualTo3TimesSum}, $i\sqrt{3} \sum\limits_{r \in G_{\frac{k}{3},3}^1(1)} (\psi_{\alpha}(x^r) - \psi_{\alpha}(-x^r))$ is a rational algebraic integer, and hence an integer for each $\alpha \in \Gamma $.
	\end{proof}
	
		\begin{lema}\label{3DividesTn(q)SameParity}
	Let $\Gamma$ be a finite abelian group and $x\in \Gamma(3)$. Then $C_x(\alpha)$ and $\frac{T_x(\alpha)}{3}$ are integers of the same parity for each $\alpha \in \Gamma $.
	\end{lema}
	\begin{proof}
	Let $x\in \Gamma(3)$ and $\alpha \in \Gamma $. By Lemma~\ref{integerEigenvalue}, $T_x(\alpha)+ C_x(\alpha)=  2 Z_x(\alpha)$ is an even integer, therefore $T_x(\alpha)$ and $C_x(\alpha)$ are integers of the same parity. By Lemma~\ref{3DividesTn(q)}, $\frac{T_x(\alpha)}{3}$ is an integer. Hence $C_x(\alpha)$ and $\frac{T_x(\alpha)}{3}$ are integers of the same parity.
	\end{proof}
	
Let $S$ be a subset of $\Gamma$. For each $\alpha \in \Gamma$, define
$$f_{\alpha}(S) = \sum_{s \in S \setminus \overline{S}} \psi_{\alpha}(s)  \hspace{0.5cm}\textnormal{and}\hspace{0.5cm} g_{\alpha}(S)= \sum_{s \in \overline{S}}(\omega \psi_{\alpha}(s) + \overline{w}\psi_{\alpha}(-s)),$$ 
where $\omega=\frac{1}{2} - \frac{i\sqrt{3}}{6}$. It is clear that $f_{\alpha}(S)$ and $g_{\alpha}(S)$ are real numbers. We have 
$$\sum_{s \in S} \psi_{\alpha}(s)= f_{\alpha}(S)+ g_{\alpha}(S) + \left( \frac{-1}{2} + \frac{i\sqrt{3}}{2} \right)( g_{\alpha}(S) - g_{-\alpha}(S) ).$$ 
Note that $f_{\alpha}(S)=f_{-\alpha}(S)$ for each $\alpha \in \Gamma$. Therefore if $f_{\alpha}(S) + g_{\alpha}(S)$ is an integer for each $\alpha\in \Gamma$, then $g_{\alpha}(S)-g_{-\alpha}(S)= \left[ f_{\alpha}(S)+g_{\alpha}(S)\right]-\left[ f_{-\alpha}(S)+g_{-\alpha}(S)\right]$ is also an integer for each $\alpha \in \Gamma$. Hence, the mixed Cayley graph $\text{Cay}(\Gamma,S)$ is Eisenstein integral if and only if $f_{\alpha}(S) + g_{\alpha}(S)$ is an integer for each $\alpha \in \Gamma$.

\begin{lema}\label{CharaEisensteinIntegral}
Let $S$ be a subset of a finite abelian group $\Gamma$ with $0 \not\in S$. Then the mixed Cayley graph $\text{Cay}(\Gamma,S)$ is Eisenstein integral if and only if $2 f_{\alpha}(S)$ and $2 g_{\alpha}(S)$ are integers of the same parity for each $\alpha \in \Gamma$.
\end{lema}
\begin{proof} 
Suppose the mixed Cayley graph $\text{Cay}(\Gamma,S)$ is Eisenstein integral and $\alpha \in \Gamma$. Then $f_{\alpha}(S) + g_{\alpha}(S)$ and $g_{\alpha}(S)-g_{-\alpha}(S)= \sum\limits_{s\in \overline{S}} \frac{-1}{3}\left[i\sqrt{3}\left(\psi_{\alpha}(s)-\psi_{\alpha}(-s)\right)\right]$ are integers. By Lemma~\ref{Sqrt3NecessIntSum}, $\sum\limits_{s\in \overline{S}\cup \overline{S}^{-1}} \psi_{\alpha}(s) \in \mathbb{Z}$. Since 
$$2 g_{\alpha}(S)= \sum\limits_{s\in \overline{S}\cup \overline{S}^{-1}} \psi_{\alpha}(s)-  \sum\limits_{s\in \overline{S}}\frac{i\sqrt{3}}{3} (\psi_{\alpha}(s)- \psi_{\alpha}(-s)),$$ we find that $2 g_{\alpha}(S)$ is an integer. Therefore, $2 f_{\alpha}(S)=2(f_{\alpha}(S)+g_{\alpha}(S))-2g_{\alpha}(S)$ is also integer of the same parity with $g_{\alpha}(S)$.

Conversely, assume that $2 f_{\alpha}(S)$ and $2 g_{\alpha}(S)$ are integers of the same parity for each $\alpha \in \Gamma$. Then $f_{\alpha}(S) + g_{\alpha}(S)$ is an integer for each $\alpha \in \Gamma$. Hence the mixed Cayley graph $\text{Cay}(\Gamma,S)$ is Eisenstein integral.
\end{proof}

\begin{lema}\label{CharaEisensteinIntegral1}
Let $S$ be a subset of finite abelian group $\Gamma$ with $0 \not\in S$. Then the mixed Cayley graph $\text{Cay}(\Gamma,S)$ is Eisenstein integral if and only if $f_{\alpha}(S)$ and $g_{\alpha}(S)$ are integers for each $\alpha \in \Gamma$.
\end{lema}
\begin{proof}
By Lemma~\ref{CharaEisensteinIntegral}, it is enough to show that $2 f_{\alpha}(S)$ and $2 g_{\alpha}(S)$ are integers of the same parity if and only if $f_{\alpha}(S)$ and $g_{\alpha}(S)$ are integers. If $f_{\alpha}(S)$ and $g_{\alpha}(S)$ are integers, then clearly $2 f_{\alpha}(S)$ and $2 \beta_j(S)$ are even integers. Conversely, assume that $2 f_{\alpha}(S)$ and $2 g_{\alpha}(S)$ are integers of the same parity. Since $f_{\alpha}(S)$ is an algebraic integer, the integrality of $2 f_{\alpha}(S)$ implies that $f_{\alpha}(S)$ is an integer. Thus $2 f_{\alpha}(S)$ is even, and so by the assumption $2 g_{\alpha}(S)$ is also even. Hence $g_{\alpha}(S)$ is an integer.
\end{proof}
	
\begin{theorem}\label{MinCharacEisensteinInteg}
Let $S$ be a subset of a finite abelian group $\Gamma$ with $0\notin S$. Then the mixed Cayley graph $\text{Cay}(\Gamma,S)$ is Eisenstein integral if and only if $\text{Cay}(\Gamma,S)$ is HS-integral.
\end{theorem}
\begin{proof}
By Lemma~\ref{CharaEisensteinIntegral1}, it is enough to show that $f_{\alpha}(S)$ and $g_{\alpha}(S)$ are integers for each $\alpha \in \Gamma$ if and only if $\text{Cay}(\Gamma,S)$ is HS-integral. Note that $f_{\alpha}(S)$ is an eigenvalue of the Cayley graph $\text{Cay}(\Gamma,S\setminus \overline{S})$. By Theorem~\ref{Cayint}, $f_{\alpha}(S)$ is an integer for each $\alpha \in \Gamma$ if and only if $S\setminus \overline{S}\in \mathbb{B}(\Gamma)$.

Assume that $f_{\alpha}(S)$ and $g_{\alpha}(S)$ are integers for each $\alpha \in \Gamma$. Then $ \sum\limits_{s\in \overline{S}} \frac{-i\sqrt{3}}{3}( \psi_{\alpha}(s)- \psi_{\alpha}(-s))=g_{\alpha}(S)-g_{-\alpha}(S)$ is also an integer for each $\alpha \in \Gamma$. Using Theorem~\ref{Cayint} and Lemma~\ref{CharaNewIntegSum}, we see that $S \setminus \overline{S}$ and $\overline{S}$ satisfy the conditions of Theorem~\ref{CharaHSintMixed}. Hence $\text{Cay}(\Gamma,S)$ is HS-integral.

Conversely, assume that $\text{Cay}(\Gamma,S)$ is HS-integral. Then $\text{Cay}(\Gamma,S\setminus \overline{S})$ is integral, and hence $f_{\alpha}(S)$ is an integer for each $\alpha \in \Gamma$. By Theorem~\ref{CharaHSintMixed}, we have $\overline{S} \in \mathbb{E}(\Gamma)$, and so $\overline{S}= \langle\!\langle x_1 \rangle\!\rangle \cup ...\cup \langle\!\langle x_k \rangle\!\rangle $ for some $x_1,...,x_k \in \Gamma (3)$. Then
	\begin{equation*}
		\begin{split}
g_{\alpha}(S)&=\frac{1}{2} \sum\limits_{s\in \overline{S}\cup \overline{S}^{-1}}\psi_{\alpha}(s) -\frac{1}{6} \sum\limits_{s\in \overline{S}}i \sqrt{3}\left(\psi_{\alpha}(s)-\psi_{\alpha}(-s)\right)\\
&=\frac{1}{2} \sum\limits_{j=1}^{k} \sum\limits_{s\in [ x_j ] }\psi_{\alpha}(s) -\frac{1}{6} \sum\limits_{j=1}^{k} \sum\limits_{s\in \langle\!\langle x_j \rangle\!\rangle}i \sqrt{3}\left(\psi_{\alpha}(s)-\psi_{\alpha}(-s)\right)\\
&= \frac{1}{2} \sum\limits_{j=1}^{k} C_{x_j}(\alpha) - \frac{1}{6} \sum\limits_{j=1}^{k} T_{x_j}(\alpha)\\
&= \frac{1}{2} \sum\limits_{j=1}^{k} \left( C_{x_j}(\alpha) - \frac{1}{3} T_{x_j}(\alpha) \right).
		\end{split} 
	\end{equation*}
By Lemma~\ref{3DividesTn(q)SameParity}, $C_{x_j}(\alpha) - \frac{1}{3} T_{x_j}(\alpha)$ is an even integer for each $j\in \{1,\ldots,k\}$. Hence $g_{\alpha}(S)$ is an integer for each $\alpha \in \Gamma$.
\end{proof} 

The following example illustrates Theorem~\ref{MinCharacEisensteinInteg}.
\begin{ex}
Consider the HS-integral graph $\text{Cay}(\mathbb{Z}_3 \times \mathbb{Z}_3, S)$ of Example~\ref{ex2}. By Theorem~\ref{MinCharacEisensteinInteg}, the graph $\text{Cay}(\mathbb{Z}_3 \times \mathbb{Z}_3, S)$ is Eisenstein integral. Indeed, the eigenvalues of  $\text{Cay}(\mathbb{Z}_3 \times \mathbb{Z}_3, S)$ are obtained as
\begin{equation*}
			\begin{split}
\gamma_{\alpha}=\psi_{\alpha}(0,1)+\psi_{\alpha}(1,0) + \psi_{\alpha}(2,0).
			\end{split}\end{equation*}
We have $\gamma_{(0,0)}=3, \gamma_{(0,1)}=2+ \omega_3, \gamma_{(0,2)}=1-\omega_3, \gamma_{(1,0)}=0, \gamma_{(1,1)}=-1+\omega_3,\gamma_{(1,2)}=-2-\omega_3, \gamma_{(2,0)}=0,$ $ \gamma_{(2,1)}=-1+\omega_3$ and $\gamma_{(2,2)}=-2-\omega_3$. Thus $\gamma_{\alpha}$ is an Eisenstein integer for each $\alpha \in \mathbb{Z}_3 \times \mathbb{Z}_3$. \qed
\end{ex}
	

\end{document}